\documentclass[a4paper,12pt, abstract=true]{scrartcl}
\usepackage[ansinew]{inputenc}
\usepackage{german}
\usepackage[german,english,french,USenglish]{babel}
\usepackage{latexsym,amssymb,amsfonts,bbm} %amsmath
\usepackage{theorem}
\usepackage{xcolor}
\usepackage{graphicx}
\usepackage{esint}
\usepackage{epstopdf}
\usepackage{stmaryrd}
\newtheorem{dummytheorem}{Dummy-Theorem}[section]
\newcommand{\proofendsign}{$\Box$} % \rule{2mm}{2mm}

\newtheorem{lemma}[dummytheorem]{Lemma}
\newtheorem{theorem}[dummytheorem]{Theorem}

\newtheorem{assumption}[dummytheorem]{Assumption}

\newenvironment{proof}{{\noindent \bf Proof }}
 {{\hspace*{\fill}\proofendsign\par\bigskip}}

\theorembodyfont{\normalfont}
\newtheorem{remarknorm}[dummytheorem]{Remark}
\newtheorem{examplenorm}[dummytheorem]{Example}

\newcommand{\N}{\mathbb{N}}

\newcommand{\R}{\mathbb{R}}

\newcommand{\pr}{\mathbb{P}}
\newcommand{\ex}{\mathbb{E}}
\newcommand{\vari}{\mathbb{V}{\rm ar}}

\newcommand{\eins}{\mathbbm{1}}

\newcommand{\vatr}{{\rm VaR}}
\newcommand{\avatr}{{\rm AVaR}}

\newcommand{\binv}{{\rm B}}

\newcommand{\law}{{\mbox{{\rm{law}}}}}

\begin{document}

%%%%%%%%%%%%%%%%%%%%%%%%%%%%%%%%%%%%%%%%%%%%%%%%%%%%%%%%%%%%%%%%
%%%%%%%%%%%%%%%%%%%%%%%%%%%%%%%%%%%%%%%%%%%%%%%%%%%%%%%%%%%%%%%%
%%%%%%%%%%%%%%%%%%%%%%%%%%%%%%%%%%%%%%%%%%%%%%%%%%%%%%%%%%%%%%%%

\title{Nonparametric estimation of risk measures of collective risks}

\author{
Alexandra Lauer\footnote{Department of Mathematics, Saarland University, {\tt lauer@math.uni-sb.de}}
\qquad\qquad
Henryk Z\"ahle\footnote{Department of Mathematics, Saarland University, {\tt zaehle@math.uni-sb.de}}}
\date{}

\maketitle

\begin{abstract}
We consider two nonparametric estimators for the risk measure of the sum of $n$ i.i.d.\ individual insurance risks where the number of historical single claims that are used for the statistical estimation is of order $n$. This framework matches the situation that nonlife insurance companies are faced with within in the scope of premium calculation. Indeed, the risk measure of the aggregate risk divided by $n$ can be seen as a suitable premium for each of the individual risks. For both estimators divided by $n$ we derive a sort of Marcinkiewicz--Zygmund strong law as well as a weak limit theorem. The behavior of the estimators for small to moderate $n$ is studied by means of Monte-Carlo simulations. \end{abstract}


%%%%%%%%%%%%%%%%%%%%%%%%%%%%%%%%%%%%%%%%%%%%%%%%%%%%%%%%%%%%%%%%
%%%%%%%%%%%%%%%%%%%%%%%%%%%%%%%%%%%%%%%%%%%%%%%%%%%%%%%%%%%%%%%%
%%%%%%%%%%%%%%%%%%%%%%%%%%%%%%%%%%%%%%%%%%%%%%%%%%%%%%%%%%%%%%%%

%%%%%%%%%%%%%%%%%%%%%%%%%%%%%%%%%%%%%%%%%%%%%%%%%%%%%%%%%%%%%%%%
%%%%%%%%%%%%%%%%%%%%%%%%%%%%%%%%%%%%%%%%%%%%%%%%%%%%%%%%%%%%%%%%
%%%%%%%%%%%%%%%%%%%%%%%%%%%%%%%%%%%%%%%%%%%%%%%%%%%%%%%%%%%%%%%%
%%%%%%%%%%%%%%%%%%%%%%%%%%%%%%%%%%%%%%%%%%%%%%%%%%%%%%%%%%%%%%%%
%%%%%%%%%%%%%%%%%%%%%%%%%%%%%%%%%%%%%%%%%%%%%%%%%%%%%%%%%%%%%%%%
%%%%%%%%%%%%%%%%%%%%%%%%%%%%%%%%%%%%%%%%%%%%%%%%%%%%%%%%%%%%%%%%

\newpage

\section{Introduction}\label{Introduction}

Let $(X_i)$ be a sequence of nonnegative i.i.d.\ random variables on a common probability space with distribution $\mu$. In the context of actuarial theory, the random variable $S_n:=\sum_{i=1}^nX_i$ can be seen as the total claim of a homogeneous insurance collective consisting of $n$ risks. The distribution of $S_n$ is given by the $n$-fold convolution $\mu^{*n}$ of $\mu$. A central task in insurance practice is the specification of the premium ${\cal R}_\rho(\mu^{*n})$ for the aggregate risk $S_n$, where ${\cal R}_\rho$ is the statistical functional associated with any suitable law-invariant risk measure $\rho$ (henceforth referred to as {\em risk functional} associated with $\rho$). Note that $\frac{1}{n}{\cal R}_\rho(\mu^{*n})$ can be seen as a suitable premium for each of the individual risks $X_1,\ldots,X_n$, where it is important to note that $\frac{1}{n}{\cal R}_\rho(\mu^{*n})$ is typically essentially smaller than ${\cal R}_\rho(\mu)$.

On the one hand, much is known about the statistical estimation of the single claim distribution $\mu$ and about the numerical approximation of the convolution $\mu^{*n}$ with known $\mu$. On the other hand, an analysis that combines both statistical aspects and the numerical approximation of $\mu^{*n}$ seems to be rare. In \cite{KraetschmerZaehle2011}, this question was approached through an estimation of $\mu^{*n}$ by the normal distribution ${\cal N}_{n\widehat m_{u_n},\,n\widehat s_{u_n}^2}$ with estimated parameters based on a sample of size $u_n\in\N$. Here $\widehat m_{u_n}$ and $\widehat s_{u_n}^2$ refer to respectively the empirical mean and the empirical variance of a sequence of i.i.d.\ random variables with distribution $\mu$ having a finite second moment. It was shown in \cite{KraetschmerZaehle2011} that for many law-invariant coherent risk measures $\rho$ and any sequence $(u_n)$ of positive integers for which $u_n/n$ converges to some constant $c\in(0,\infty)$ we have
\begin{equation}\label{convergence in context of normal approx}
    n^{r}~\frac{{\cal R}_\rho({\cal N}_{n\widehat m_{u_n},\,n\widehat s_{u_n}^2})-{\cal R}_\rho(\mu^{*n})}{n}\,\stackrel{\rm a.s.}{\longrightarrow}\, 0,\qquad n\to\infty
\end{equation}
for every $r<1/2$, and
\begin{equation}\label{weak onvergence in context of normal approx}
    \law\Big\{n^{1/2}~\frac{{\cal R}_\rho({\cal N}_{n\widehat m_{u_n},\,n\widehat s_{u_n}^2})-{\cal R}_\rho(\mu^{*n})}{n}\Big\}\,\stackrel{{\rm w}}{\longrightarrow}\,{\cal N}_{0,\,s^2},\qquad n\to\infty
\end{equation}
with $s^2:=\vari[X_1]$. Of course, (\ref{weak onvergence in context of normal approx}) implies in particular that the convergence in (\ref{convergence in context of normal approx}) cannot hold for $r\ge 1/2$. The assumption that $u_n$ increases to infinity at the same speed as $n$ increases to infinity is motivated by the fact that the parameters are typically estimated on the basis of the historical claims of the same collective from the last year or from the last few years. This is also why the presented theory is nonstandard. In the existing literature on the statistical estimation of convolutions the number of summands is typically fixed or increases essentially slower to infinity than $u_n$ does; see, for instance, \cite{Pitts1994} for the nonparametric estimation of a (compound) convolution where the (distribution of the) number of summands is fixed and known. It was also shown in \cite{KraetschmerZaehle2011} that for the exact mean $m$ and the exact variance $s^2$ of $\mu$, and for many law-invariant
coherent risk
measures $\rho$,
\begin{equation}\label{convergence in context of normal approx known param}
    \sup_{n\in\N}\,|{\cal R}_\rho({\cal N}_{nm,\,ns^2})-{\cal R}_\rho(\mu^{*n})|\,<\infty.
\end{equation}
Both (\ref{convergence in context of normal approx})--(\ref{convergence in context of normal approx known param}) and the simulation study in \cite{KraetschmerZaehle2011} show that the overwhelming part of the error in the estimated normal approximation of the risk functional is due to the estimation of the unknown parameters rather than to the numerical approximation itself. Whereas in the case of known parameters the relative error converges to zero at rate (nearly) $1$, %as $n^{-\alpha}$ for $\alpha<1$ arbitrarily close to $1$,
in the case of estimated parameters the relative error converges to zero only at rate (nearly) $1/2$. %$\frac{1}{2}$. %as $n^{-\alpha}$ only for $\alpha<1/2$.
So it is very important to note that statistical aspects may not be neglected when investigating approximations of premiums for aggregate risks.

The estimated normal approximation ${\cal R}_\rho({\cal N}_{n\widehat m_{u_n},\,n\widehat s_{u_n}^2})$ of ${\cal R}_\rho(\mu^{*n})$ is very simple and saves computing time in great measure. Indeed, we have
\begin{equation}\label{risk measure of N n}
    {\cal R}_\rho({\cal N}_{n\widehat m_{u_n},\,n\widehat s_{u_n}^2})\,=\,\sqrt{n}\widehat s_{u_n}{\cal R}_\rho({\cal N}_{0,1})+n\widehat m_{u_n}
\end{equation}
whenever ${\cal R}_\rho$ corresponds to a cash additive and positively homogeneous risk measure $\rho$. On the other hand, in real applications the total claim distribution $\mu^{*n}$ is typically skewed to the right, whereas the normal distribution is symmetric; see also Figure \ref{fig 1}. So it is natural to study methods which better fit skewed total claim distributions. In this article, we will therefore replace ${\cal N}_{n\widehat m_{u_n},\,n\widehat s_{u_n}^2}$ by the $n$-fold convolution $\widehat{\mu}_{u_n}^{\,*n}$ of the empirical estimator $\widehat{\mu}_{u_n}$ of $\mu$. The corresponding estimator ${\cal R}_\rho(\widehat{\mu}_{u_n}^{\,*n})$ will be referred to as {\em empirical plug-in estimator}. The calculation of the empirical plug-in estimator will be more computing time consuming than the calculation of the estimated normal approximation, nevertheless the needed computing time is still satisfying for actuarial applications. It is quite clear, and can also be seen from Figure \ref{fig 1},
that $\mu^{*n}$ gets increasingly skewed as the tail of $\mu$ gets heavier. So it is not surprising that the estimated normal approximation works well for light-tailed $\mu$ and gets worse for medium-tailed and heavy-tailed $\mu$. A simulation study for the Value at Risk functional in Section \ref{Numerical examples} indicates that the empirical plug-in estimator is only slightly better than the estimated normal approximation for light-tailed $\mu$ but is essentially better for medium-tailed $\mu$. For heavy-tailed $\mu$ both estimators work well only for rather large $n$. Throughout this article we will use the terms ``light-tailed'', ``medium-tailed'' and ``heavy-tailed'' in a quite sloppy way. By definition ``heavy-tailed'' refers to distributions without a finite second moment. However our theory is only applicable to distributions with a finite $\lambda$-moment for some $\lambda>2$. In this context we refer to heavy-tailed distributions whenever $\lambda$ is close to 2 and will use the terms ``medium-
tailed'' and ``light-tailed'' for  larger $\lambda$.

To introduce the empirical plug-in estimator rigorously, let $(Y_i)$ be a sequence of i.i.d.\ random variables on some probability space $(\Omega,{\cal F},\pr)$ with distribution $\mu$. The random variables $Y_i$ can be seen as observed historical single claims. The empirical probability measure of the first $u\in\N$ observations,
\begin{equation}\label{estimator for mu}
    \widehat{\mu}_{u}\,:=\,\frac{1}{u}\sum_{i=1}^{u}\delta_{Y_i},
\end{equation}
is the standard nonparametric estimator for $\mu$, and therefore
\begin{equation}\label{S n as coll model}
    \widehat\mu_{u}^{\,*n}\,:=\,(\widehat{\mu}_{u})^{*n}
\end{equation}
provides a reasonable estimator for $\mu^{*n}$. Then it is natural to use the plug-in estimator
\begin{equation}\label{def of estimator for aggregate risk}
    {\cal R}_\rho\big(\widehat \mu_{u}^{\,*n}\big)
\end{equation}
for the estimation of ${\cal R}_\rho(\mu^{*n})$; computational aspects will be discussed in the Appendix \ref{A remark on the recursive scheme}. We will see in Section \ref{main results} that for a very large class of %cash-invariant and positively homogeneous
law-invariant risk measures $\rho$, any distribution $\mu$ with a finite $\lambda$-moment for some $\lambda>2$, and any sequence $(u_n)$ of positive integers for which $u_n/n$ converges to some constant $c\in(0,\infty)$, we also have (\ref{convergence in context of normal approx})--(\ref{weak onvergence in context of normal approx}) with ${\cal N}_{n\widehat m_{u_n},\,n\widehat s_{u_n}^2}$ replaced by $\widehat\mu_{u_n}^{\,*n}$. In Theorems \ref{main theorem} and \ref{main theorem - EPiE} we will prove even more, namely
\begin{eqnarray}
    \frac{{\cal R}_\rho({\cal N}_{n\widehat m_{u_n},\,n\widehat s_{u_n}^2})-{\cal R}_\rho(\mu^{*n})}{n} & = & (\widehat m_{u_n}-m)\,+\,o_{\mbox{\rm \scriptsize{$\pr$-a.s.}}}(n^{-1/2}),\label{convergence in context of recursive method - prime - 1}\\
    \frac{{\cal R}_\rho(\widehat\mu_{u_n}^{\,*n})-{\cal R}_\rho(\mu^{*n})}{n} & = & (\widehat m_{u_n}-m)\,+\,o_{\mbox{\rm \scriptsize{$\pr$-a.s.}}}(n^{-1/2}),\label{convergence in context of recursive method - prime - 2}
\end{eqnarray}
where $o_{\mbox{\rm \scriptsize{$\pr$-a.s.}}}(n^{-1/2})$ refers to any sequence of random variables $(\xi_n)$ on $(\Omega,{\cal F},\pr)$ for which $\sqrt{n}\xi_n$ converges $\pr$-a.s.\ to zero as $n\rightarrow\infty$. Assertions (\ref{convergence in context of recursive method - prime - 1})--(\ref{convergence in context of recursive method - prime - 2}) have an astonishing consequence. {\em No matter what the particular risk measure $\rho$ looks like}, the asymptotics of the estimators $\frac{1}{n}{\cal R}_\rho({\cal N}_{n\widehat m_{u_n},\,n\widehat s_{u_n}^2})$ and $\frac{1}{n}{\cal R}_\rho(\widehat\mu_{u_n}^{\,*n})$ for the individual premium $\frac{1}{n}{\cal R}_\rho(\mu^{*n})$ are exactly the same as for the empirical mean regarded as an estimator for the mean. By the classical Central Limit Theorem, we can derive from (\ref{convergence in context of recursive method - prime - 1}) and (\ref{convergence in context of recursive method - prime - 2})
the following asymptotic confidence intervals at level $1-\alpha$ for the individual premium $\frac{1}{n}{\cal R}_\rho(\mu^{*n})$:
$$
    \bigg[\frac{{\cal R}_\rho({\cal N}_{n\widehat m_{u_n},\,n\widehat s_{u_n}^2})}{n}-\frac{\widehat s_{u_n}}{\sqrt{n}}\,\Phi_{0,1}^{-1}\Big(1-\frac{\alpha}{2}\Big)\,,\,\frac{{\cal R}_\rho({\cal N}_{n\widehat m_{u_n},\,n\widehat s_{u_n}^2})}{n}-\frac{\widehat s_{u_n}}{\sqrt{n}}\,\Phi_{0,1}^{-1}\Big(\frac{\alpha}{2}\Big)\bigg]
$$
and
$$
    \bigg[\frac{{\cal R}_\rho(\widehat\mu_{u_n}^{\,*n})}{n}-\frac{\widehat s_{u_n}}{\sqrt{n}}\,\Phi_{0,1}^{-1}\Big(1-\frac{\alpha}{2}\Big)\,,\,\frac{{\cal R}_\rho(\widehat\mu_{u_n}^{\,*n})}{n}-\frac{\widehat s_{u_n}}{\sqrt{n}}\,\Phi_{0,1}^{-1}\Big(\frac{\alpha}{2}\Big)\bigg],
$$
where $\Phi_{0,1}$ denotes the distribution function of ${\cal N}_{0,1}$.

Further, it is a simple consequence of part (ii) of Theorem \ref{main theorem} below that
\begin{equation}\label{convergence in context of recursive method - prime - 5}
    \frac{{\cal R}_\rho(\mu^{*n})}{n}\,=\,m\,+\,\frac{{\cal R}_\rho({\cal N}_{0,1})}{\sqrt{n}}\,s\,+{\cal O}(n^{-1/2-\gamma\beta})
\end{equation}
with $\gamma:=\min\{\lambda-2;1\}/2$ and $\beta>0$ depending on $\rho$. The identity (\ref{convergence in context of recursive method - prime - 5}) shows that for large $n$ (and $\beta$ and $\gamma$ away from $0$) the individual premium $\frac{1}{n}{\cal R}_\rho(\mu^{*n})$ can be seen as an approximation of the premium which is determined according to the standard deviation principle with safety loading $\frac{1}{\sqrt{n}}{\cal R}_\rho({\cal N}_{0,1})$. For the corresponding estimators we will obtain (cf.\ Remark \ref{remark on representation} below) the following empirical analogues of (\ref{convergence in context of recursive method - prime - 5}):
\begin{eqnarray}
    \frac{{\cal R}_\rho({\cal N}_{n\widehat m_{u_n},\,n\widehat s_{u_n}^2})}{n} & = & \widehat m_{u_n}\,+\,\frac{{\cal R}_\rho({\cal N}_{0,1})}{\sqrt{n}}\,\widehat s_{u_n},
    \nonumber\\
    \frac{{\cal R}_\rho(\widehat\mu_{u_n}^{\,*n})}{n} & = & \widehat m_{u_n}\,+\,\frac{{\cal R}_\rho({\cal N}_{0,1})}{\sqrt{n}}\,\widehat s_{u_n}\,+\,{\cal O}_{\mbox{\rm \scriptsize{$\pr$-a.s.}}}(n^{-1/2-\gamma\beta}),\label{convergence in context of recursive method - prime - 20}
\end{eqnarray}
where ${\cal O}_{\mbox{\rm \scriptsize{$\pr$-a.s.}}}(n^{-1/2-\gamma\beta})$ refers to any sequence of random variables $(\xi_n)$ on $(\Omega,{\cal F},\pr)$ for which the sequence $(n^{1/2+\gamma\beta}\xi_n)$ is
bounded $\pr$-a.s. To some extent, (\ref{convergence in context of recursive method - prime - 5}) and (\ref{convergence in context of recursive method - prime - 20}) justify the use of the standard deviation principle (with $m$ and $s$ estimated by $\widehat m_{u_n}$ and $\widehat s_{u_n}$, respectively), which many insurance companies use to determine individual premiums in large collectives. In practice the specific choice of the safety loading in the context of the standard deviation principle is often somewhat arbitrary. Formulae (\ref{convergence in context of recursive method - prime - 5}) and (\ref{convergence in context of recursive method - prime - 20}) now give a deeper insight into the practical choice of the safety loading. It should be chosen as the product of a suitable risk functional (which one
has actually in mind) evaluated at the standard normal distribution and the factor $1/{\sqrt{n}}$ (where $n$ is the size of the collective). The factor $1/{\sqrt{n}}$ reflects the balancing of risks in (large) collectives.

It is quite clear that the goodness of the estimator in (\ref{def of estimator for aggregate risk}) can be improved through replacing the nonparametric estimator $\widehat\mu_u$ in (\ref{estimator for mu})--(\ref{S n as coll model}) by a suitable estimator that is based on a parametric statistical model. However, this requires preliminary considerations w.r.t.\ a proper choice of the parametric model. Such considerations are feasible and common. Nevertheless we leave the parametric approach for future work.

The rest of the article is organized as follows. In Section \ref{main results} we will present our results, and in Sections \ref{Illustration of Assumption} and \ref{Numerical examples} these results will be illustrated by means of examples. The results of Section \ref{main results} will be proven in Section \ref{Proofs}. A remark on the computation of the empirical plug-in estimator can be found in the Appendix \ref{A remark on the recursive scheme}.
%of Theorem \ref{main theorem} relies on a new nonuniform Berry--Esséen inequality. This inequality will be presented in Section \ref{Berry-Esseen} and is of independent interest. From a mathematical point of view the proof of the nonuniform Berry--Esséen inequality is the hard part about this article.

%%%%%%%%%%%%%%%%%%%%%%%%%%%%%%%%%%%%%%%%%%%%%%%%%%%%%%%%%%%%%%%%
%%%%%%%%%%%%%%%%%%%%%%%%%%%%%%%%%%%%%%%%%%%%%%%%%%%%%%%%%%%%%%%%
%%%%%%%%%%%%%%%%%%%%%%%%%%%%%%%%%%%%%%%%%%%%%%%%%%%%%%%%%%%%%%%%
%%%%%%%%%%%%%%%%%%%%%%%%%%%%%%%%%%%%%%%%%%%%%%%%%%%%%%%%%%%%%%%%
%%%%%%%%%%%%%%%%%%%%%%%%%%%%%%%%%%%%%%%%%%%%%%%%%%%%%%%%%%%%%%%%
%%%%%%%%%%%%%%%%%%%%%%%%%%%%%%%%%%%%%%%%%%%%%%%%%%%%%%%%%%%%%%%%

\section{Results}\label{main results}

Let $L^0$ denote the usual set of all finitely-valued random variables on an atomless probability space modulo the equivalence relation of almost sure identity. Let ${\cal X}\subset L^0$ be a vector space containing the constants. An intrinsic example for ${\cal X}$ is the space $L^p$ (consisting of all $p$-fold integrable random variables from $L^0$) for $p\ge 1$. We will say that a map $\rho:{\cal X}\rightarrow\R$ is
\begin{itemize}
    \item {\em monotone} if $\rho(X_1)\le\rho(X_2)$ for all $X_1,X_2\in{\cal X}$ with $X_1\le X_2$.
    \item {\em cash additive} if $\rho(X+m)=\rho(X)+m$ for all $X\in{\cal X}$ and $m\in\R$.
    \item {\em subadditive} if $\rho(X_1+X_2)\le\rho(X_1)+\rho(X_2)$ for all $X_1,X_2\in{\cal X}$.
    \item {\em positively homogenous} if $\rho(\lambda X)=\lambda\rho(X)$ for all $X\in{\cal X}$ and $\lambda\ge 0$.
\end{itemize}
As usual, we will say that $\rho$ is {\em coherent} if it satisfies all of these four conditions, and that $\rho$ is {\em law-invariant} if $\rho(X)=\rho(Y)$ whenever $X$ and $Y$ have the same law. We will restrict ourselves to law-invariant maps $\rho:{\cal X}\rightarrow\R$. So we may and do associate with $\rho$ a statistical functional ${\cal R}_\rho:{\cal M}({\cal X})\rightarrow\R$ via
\begin{equation}\label{def risk functional}
    {\cal R}_\rho(\mu)\,:=\,\rho(X_\mu),\qquad \mu\in{\cal M}({\cal X}),
\end{equation}
where ${\cal M}({\cal X})$ denotes the set of the distributions of the elements of ${\cal X}$, and $X_\mu\in{\cal X}$ has distribution $\mu$.

Let ${\cal M}_1$ be the set of all probability measures on $(\R,{\cal B}(\R))$, and denote by $F_\mu$ the distribution function of $\mu\in{\cal M}_1$. For every $\lambda\ge 0$, let the function $\phi_\lambda:\R\rightarrow[1,\infty)$ be defined by $\phi_\lambda(x):=(1+|x|^\lambda)$, $x\in\R$. For $\mu_1,\mu_2\in{\cal M}_1$, we say that
\begin{equation}\label{weighted kolmogorov distance}
    d_{\phi_\lambda}(\mu_1,\mu_2)\,:=\,\sup_{x\in\R}\,|F_{\mu_1}(x)-F_{\mu_2}(x)|\,\phi_\lambda(x)
\end{equation}
is the nonuniform Kolmogorov distance of $\mu_1$ and $\mu_2$ w.r.t.\ the weight function $\phi_\lambda$. It is easily seen that $d_{\phi_\lambda}$ provides a metric on the set ${\cal M}_1^\lambda$ of all $\mu\in{\cal M}_1$ satisfying $d_{\phi_\lambda}(\mu,\delta_0)<\infty$.

Recall that $(Y_i)$ is a sequence of i.i.d.\ real-valued random variables on some probability space $(\Omega,{\cal F},\pr)$ with distribution $\mu$ having a finite second moment, and that the estimators $\widehat{\mu}_{u}$, $\widehat\mu_{u}^{\,*n}$, and ${\cal R}_\rho(\widehat \mu_{u}^{\,*n})$ are given by (\ref{estimator for mu}), (\ref{S n as coll model}), and (\ref{def of estimator for aggregate risk}), respectively. We set $m:=\ex[Y_1]$ and $s:=\vari[Y_1]^{1/2}$, and let $\widehat m_{u}:=\frac{1}{u}\sum_{i=1}^uY_i$ and $\widehat s_{u}:=(\frac{1}{u}\sum_{i=1}^u(Y_i-\widehat m_{u})^2)^{1/2}$ be the corresponding standard nonparametric estimators. The following Assumption \ref{basic assumption} will be illustrated in Section \ref{Illustration of Assumption}. 

\begin{assumption}\label{basic assumption}
Let $\rho:{\cal X}\rightarrow\R$ be a law-invariant map, and ${\cal R}_\rho$ be the corresponding statistical functional introduced in (\ref{def risk functional}). Let $(u_n)$ be a sequence in $\N$, and assume that the following assertions hold for some $\lambda>2$:
\begin{itemize}
    \item[(a)] $\mu\in{\cal M}(L^\lambda)$, that is, $\ex\big[|Y_1|^{\lambda}\big]<\infty$.
    \item[(b)] $u_n/n$ converges to some constant $c\in(0,\infty)$.
    \item[(c)] $\rho$ is cash additive and positively homogeneous, and ${\cal M}_1^\lambda\subset{\cal M}({\cal X})$.
    \item[(d)] For each sequence $(\mathfrak{m}_n)\subset{\cal M}_1^\lambda$ with $d_{\phi_\lambda}(\mathfrak{m}_n,{\cal N}_{0,1})\rightarrow 0$, there exist constants $C,\beta>0$
    such that $|{\cal R}_\rho(\mathfrak{m}_n)-{\cal R}_\rho({\cal N}_{0,1})|\le C d_{\phi_\lambda}(\mathfrak{m}_n,{\cal N}_{0,1})^{\beta}$ for all $n\in\N$.
\end{itemize}
\end{assumption}

The following result is basically already known from \cite{KraetschmerZaehle2011}. Assertions (iv)--(v) in Theorem \ref{main theorem} describe the asymptotic behavior of the estimator $\frac{1}{n}{\cal R}_\rho({\cal N}_{n\widehat m_{u_n},\,n\widehat s_{u_n}^2})$ for the individual premium $\frac{1}{n}{\cal R}_\rho(\mu^{*n})$. Note that ${\cal R}_\rho({\cal N}_{n\widehat m_{u_n},\,n\widehat s_{u_n}^2})$ is always $({\cal F},{\cal B}(\R))$-measurable due to the representation (\ref{risk measure of N n}).

\begin{theorem}\label{main theorem}{\bf (Estimated normal approximation)}
Suppose that Assumption \ref{basic assumption} holds with $\lambda>2$ and $\beta>0$, and let $\gamma:=\min\{\lambda-2;1\}/2$. Then the following assertions hold:
\begin{itemize}
    \item[(i)] $\frac{1}{n}({\cal R}_\rho({\cal N}_{n\widehat m_{u_n},\,n\widehat s_{u_n}^2})-{\cal R}_\rho({\cal N}_{nm,\,ns^2}))=(\widehat m_{u_n}-m)+o_{\mbox{\rm \scriptsize{$\pr$-a.s.}}}(n^{-1/2})$.
    \item[(ii)] $\frac{1}{n}({\cal R}_\rho({\cal N}_{nm,\,ns^2})-{\cal R}_\rho(\mu^{*n}))={\cal O}(n^{-1/2-\gamma \beta})$.
    \item[(iii)] $\frac{1}{n}({\cal R}_\rho({\cal N}_{n\widehat m_{u_n},\,n\widehat s_{u_n}^2})-{\cal R}_\rho(\mu^{*n}))=(\widehat m_{u_n}-m)+o_{\mbox{\rm \scriptsize{$\pr$-a.s.}}}(n^{-1/2})$.
    \item[(iv)] $n^{r}\frac{1}{n}({\cal R}_\rho({\cal N}_{n\widehat m_{u_n},\,n\widehat s_{u_n}^2})-{\cal R}_\rho(\mu^{*n}))\longrightarrow0$ $\pr$-a.s.\ for every $r<1/2$.
    \item[(v)] $\mbox{\rm law}\big\{\sqrt{n}\,\frac{1}{n}({\cal R}_\rho({\cal N}_{n\widehat m_{u_n},\,n\widehat s_{u_n}^2})-{\cal R}_\rho(\mu^{*n}))\big\}\stackrel{{\rm w}}{\longrightarrow}{\cal N}_{0,\,s^2}$.
\end{itemize}
\end{theorem}

The following result provides the analogue of Theorem \ref{main theorem} for the empirical plug-in estimator $\frac{1}{n}{\cal R}_\rho(\widehat\mu_{u_n}^{\,*n})$ for the individual premium $\frac{1}{n}{\cal R}_\rho(\mu^{*n})$. Assertions (iii)--(iv) in Theorem \ref{main theorem - EPiE} describe the asymptotic behavior of the estimator $\frac{1}{n}{\cal R}_\rho(\widehat\mu_{u_n}^{\,*n})$.

\begin{theorem}\label{main theorem - EPiE}{\bf (Empirical plug-in estimator)}
Suppose that Assumption \ref{basic assumption} holds with $\lambda>2$ and $\beta>0$, let $\gamma:=\min\{\lambda-2;1\}/2$, and assume that ${\cal R}_\rho(\widehat\mu_{u_n}^{\,*n})$ is $({\cal F},{\cal B}(\R))$-measurable for every $n\in\N$. Then the following assertions hold:
\begin{itemize}
    \item[(i)] $\frac{1}{n}({\cal R}_\rho({\cal N}_{n\widehat m_{u_n},\,n\widehat s_{u_n}^2})-{\cal R}_\rho(\widehat\mu_{u_n}^{\,*n}))={\cal O}_{\mbox{\rm \scriptsize{$\pr$-a.s.}}}(n^{-1/2-\gamma \beta})$.
    \item[(ii)]$\frac{1}{n}({\cal R}_\rho(\widehat\mu_{u_n}^{\,*n}))-{\cal R}_\rho(\mu^{*n}))=(\widehat m_{u_n}-m)+o_{\mbox{\rm \scriptsize{$\pr$-a.s.}}}(n^{-1/2})$.
    \item[(iii)] $n^{r}\frac{1}{n}({\cal R}_\rho(\widehat\mu_{u_n}^{\,*n}))-{\cal R}_\rho(\mu^{*n}))\longrightarrow0$ $\pr$-a.s.\ for every $r<1/2$.
    \item[(iv)] $\mbox{\rm law}\big\{\sqrt{n}\,\frac{1}{n}({\cal R}_\rho(\widehat\mu_{u_n}^{\,*n}))-{\cal R}_\rho(\mu^{*n}))\big\}\stackrel{{\rm w}}{\longrightarrow}{\cal N}_{0,\,s^2}$.
\end{itemize}
\end{theorem}

Note that the measurability assumption on ${\cal R}_\rho(\widehat\mu_{u_n}^{\,*n})$ in Theorem \ref{main theorem - EPiE} is not very restrictive. For instance, when $\rho$ is the Value at Risk or a distortion risk measure (see Sections \ref{Sec VaR}--\ref{Sec DRM} for details), then it can be easily seen that ${\cal R}_\rho(\widehat\mu_{u_n}^{\,*n})$ is $({\cal F},{\cal B}(\R))$-measurable. Moreover, the measurability also holds for any law-invariant coherent risk measure $\rho$ which is defined on $L^p$ for some $p\in[1,\infty)$:

\begin{remarknorm}\label{remark on measurability}
Let ${\cal X}=L^p$ for some $p\in[1,\infty)$. Then for every law-invariant coherent risk measure $\rho:L^p\rightarrow\R$ the estimator ${\cal R}_\rho(\widehat\mu_{u_n}^{\,*n})$ is $({\cal F},{\cal B}(\R))$-measurable for every $n\in\N$. See
Section \ref{proof of remark on measurability} for a verifications.
{\hspace*{\fill}$\Diamond$\par\bigskip}
\end{remarknorm}

\begin{remarknorm}\label{remark on representation}
Note that the considered estimators for the individual premium have the following representations:
\begin{eqnarray}
    \frac{1}{n}{\cal R}_\rho({\cal N}_{n\widehat m_{u_n},\,n\widehat s_{u_n}^2}) & = & \widehat m_{u_n}+\frac{1}{\sqrt{n}}\,\widehat s_{u_n}{\cal R}_\rho({\cal N}_{0,1}),\label{remark on representation - 10}\\
    \frac{1}{n}{\cal R}_\rho(\widehat\mu_{u_n}^{\,*n}) & = & \widehat m_{u_n}+\frac{1}{\sqrt{n}}\,\widehat s_{u_n}{\cal R}_\rho({\cal N}_{0,1})+{\cal O}_{\mbox{\rm \scriptsize{$\pr$-a.s.}}}(n^{-1/2-\gamma \beta}).\label{remark on representation - 20}
\end{eqnarray}
Equation (\ref{remark on representation - 10}) is a simple consequence of part (c) of Assumption \ref{basic assumption}, and (\ref{remark on representation - 20}) follows from (\ref{remark on representation - 10}) and part (i) of Theorem
\ref{main theorem - EPiE}.
{\hspace*{\fill}$\Diamond$\par\bigskip}
\end{remarknorm}

%%%%%%%%%%%%%%%%%%%%%%%%%%%%%%%%%%%%%%%%%%%%%%%%%%%%%%%%%%%%%%%%
%%%%%%%%%%%%%%%%%%%%%%%%%%%%%%%%%%%%%%%%%%%%%%%%%%%%%%%%%%%%%%%%
%%%%%%%%%%%%%%%%%%%%%%%%%%%%%%%%%%%%%%%%%%%%%%%%%%%%%%%%%%%%%%%%
%%%%%%%%%%%%%%%%%%%%%%%%%%%%%%%%%%%%%%%%%%%%%%%%%%%%%%%%%%%%%%%%
%%%%%%%%%%%%%%%%%%%%%%%%%%%%%%%%%%%%%%%%%%%%%%%%%%%%%%%%%%%%%%%%
%%%%%%%%%%%%%%%%%%%%%%%%%%%%%%%%%%%%%%%%%%%%%%%%%%%%%%%%%%%%%%%%

\section{Illustration of Assumption \ref{basic assumption}}\label{Illustration of Assumption}

%%%%%%%%%%%%%%%%%%%%%%%%%%%%%%%%%%%%%%%%%%%%%%%%%%%%%%%%%%%%%%%%
%%%%%%%%%%%%%%%%%%%%%%%%%%%%%%%%%%%%%%%%%%%%%%%%%%%%%%%%%%%%%%%%
%%%%%%%%%%%%%%%%%%%%%%%%%%%%%%%%%%%%%%%%%%%%%%%%%%%%%%%%%%%%%%%%

\subsection{Value at Risk}\label{Sec VaR}

The {\em Value at Risk} at level $\alpha\in(0,1)$ is the map $\vatr_\alpha:L^0\rightarrow\R$ defined by
$$
    \vatr_\alpha(X)\,:=\,F_X^{\leftarrow}(\alpha)\,:=\,\inf\{x\in\R:\,F_X(x)\ge\alpha\}.
$$
It is clearly law-invariant and easily seen to be monotone, cash additive, and positively homogeneous. Moreover, ${\cal M}_1^\lambda\subset{\cal M}(L^0)$ trivially holds for every $\lambda\ge 0$. In particular, it satisfies condition (c) of
Assumption \ref{basic assumption}. It follows from Theorem 2 in \cite{Zaehle2011} that $\vatr_\alpha$ also satisfies condition (d) of Assumption \ref{basic assumption} for $\beta=1$ and every $\lambda\ge 0$.

%%%%%%%%%%%%%%%%%%%%%%%%%%%%%%%%%%%%%%%%%%%%%%%%%%%%%%%%%%%%%%%%%
%%%%%%%%%%%%%%%%%%%%%%%%%%%%%%%%%%%%%%%%%%%%%%%%%%%%%%%%%%%%%%%%%
%%%%%%%%%%%%%%%%%%%%%%%%%%%%%%%%%%%%%%%%%%%%%%%%%%%%%%%%%%%%%%%%%
%
\subsection{Coherent distortion risk measure}\label{Sec DRM}

Let $g:[0,1]\to[0,1]$ be a convex distortion function, i.e.\ a convex and nondecreasing function with $g(0)=0$ and $g(1)=1$. Note that the function $g$ is continuous on $[0,1)$ and might jump at $1$. The {\em distortion risk measure} associated with $g$ is defined by
\begin{equation}\label{Def DR Measure}
    \rho_g(X)\,:=\,-\int_{-\infty}^0 g(F_X(x))\,dx+\int_0^\infty \big(1-g(F_X(x))\big)\,dx
\end{equation}
for every real-valued random variable $X$ (on some given atomless probability space) satisfying $\int_0^\infty (1-g(F_{|X|}(x))dx<\infty$, where $F_X$ and $F_{|X|}$ denote the distribution functions of $X$ and $|X|$, respectively. The set ${\cal X}_g\subset L^0$ of all such random variables forms a linear subspace of $L^1$; this follows from \cite[Proposition 9.5]{Denneberg1994} and \cite[Proposition 4.75]{FoellmerSchied2011}. It is known that $\rho_g$ is a law-invariant coherent risk measure; see, for instance, \cite{WangDhaene1998}.

\begin{lemma}\label{prop on drm}
Let $\rho_g:{\cal X}_g\rightarrow\R$ be the distortion risk measure associated with a convex distortion function $g$. Assume that there exist constants $L,\beta>0$ such that
\begin{equation}\label{prop on drm - eq}
    1-g(t)\le L (1-t)^\beta\quad\mbox{ for all }t\in[0,1].
\end{equation}
Then $\rho_g$ satisfies conditions (c)--(d) of Assumption \ref{basic assumption} for this $\beta$ and every $\lambda>0$ with $\lambda \beta>1$.
\end{lemma}

\begin{proof}
The first part of condition (c) is satisfied since $\rho_g$ is a law-invariant coherent risk measure. Condition (\ref{prop on drm - eq}) and the convexity of the distortion function $g$ together imply $|g(t)-g(t')\le L|t-t'|^\beta$ for all
$t,t'\in[0,1]$. In view of (\ref{Def DR Measure}) and the assumption $\lambda \beta>1$, it follows easily that condition (d) and the second part of condition (c) hold too.
\end{proof}

If specifically $g(t)=\frac{1}{1-\alpha}\max\{t-\alpha;0\}$ for any fixed $\alpha\in(0,1)$, then we have ${\cal X}_g=L^1$ and $\rho_g$ is nothing but the {\em Average Value at Risk} $\avatr_\alpha$ at level $\alpha$. In this case, condition
(\ref{prop on drm - eq}) holds for $\beta=1$. That is, $\avatr_\alpha$ satisfies conditions (c)--(d) of Assumption \ref{basic assumption} for every $\lambda>1$.

%%%%%%%%%%%%%%%%%%%%%%%%%%%%%%%%%%%%%%%%%%%%%%%%%%%%%%%%%%%%%%%%
%%%%%%%%%%%%%%%%%%%%%%%%%%%%%%%%%%%%%%%%%%%%%%%%%%%%%%%%%%%%%%%%
%%%%%%%%%%%%%%%%%%%%%%%%%%%%%%%%%%%%%%%%%%%%%%%%%%%%%%%%%%%%%%%%

\subsection{Further coherent risk measures}

Not every law-invariant coherent risk measure $\rho$ can be seen as a distortion risk measure. In particular, Lemma \ref{prop on drm} is not a general device to verify condition (d) of Assumption \ref{basic assumption}. If $\rho$ is not a distortion risk measure, then the following Lemma \ref{prop on Orlicz rm} might help.

\begin{lemma}\label{prop on Orlicz rm}
Let $p\ge 1$. Let $\rho:L^p\rightarrow\R$ be a law-invariant coherent risk measure and
%Assume that the standard normal distribution ${\cal N}_{0,1}$ lies in ${\cal M}(H^\Psi)$.
define a function $g_\rho:[0,1]\rightarrow[0,1]$ by $g_\rho(t)\,:=\,1-\rho(B_{1-t})$, where $B_{1-t}$ refers to any Bernoulli random variable with expectation $1-t$. Assume that there exist constants $L,\beta>0$ such that
\begin{equation}\label{prop on Orlicz rm - eq}
    1-g_\rho(t)\le L (1-t)^\beta\quad\mbox{ for all }t\in[0,1].
\end{equation}
Then $\rho_g$ satisfies conditions (c)--(d) of Assumption \ref{basic assumption} for this $\beta$ and every $\lambda>p$ with $\lambda \beta>1$.
\end{lemma}

\begin{proof}
The assumption $\lambda>p$ ensures ${\cal M}_1^\lambda\subset{\cal M}(L^p)$, so that condition (c) is satisfied. Since $\rho$ is defined on $L^p$, we can find a set ${\cal G}_\rho$ of continuous convex distortion functions such that
$g_\rho=\inf_{g\in{\cal G}_\rho}g$ and
\begin{equation}\label{prop on Orlicz rm - PROOF - 10}
    \rho(X)=\sup_{g\in{\cal G}_\rho}\rho_g(X)\quad\mbox{ for all }X\in L^p.
\end{equation}
This follows from Proposition 5.1 and Remark 3.2 in \cite{BelomestnyKraetschmer2012} (adapted to our definition of monotonicity and cash additivity); see also \cite{KraetschmerZaehle2011,Kraetschmeretal2015}. Below we will show that (\ref{prop on Orlicz rm - eq}) implies
\begin{equation}\label{prop on Orlicz rm - PROOF - 20}
    |g(t)-g(t')|\le L |t-t'|^\beta\quad\mbox{ for all }t,t'\in[0,1]\mbox{ and }g\in{\cal G}_\rho.
\end{equation}
With the help of (\ref{prop on Orlicz rm - PROOF - 10}) and (\ref{prop on Orlicz rm - PROOF - 20}) we then obtain
\begin{eqnarray*}
    |{\cal R}_\rho(\mathfrak{m}_n)-{\cal R}_\rho({\cal N}_{0,1})|
    & = & \Big|\sup_{g\in{\cal G}_\rho}{\cal R}_{\rho_g}(\mathfrak{m}_n)-\sup_{g\in{\cal G}_\rho}{\cal R}_{\rho_g}({\cal N}_{0,1})\Big|\\
    & \le & \sup_{g\in{\cal G}_\rho}\big|{\cal R}_{\rho_g}(\mathfrak{m}_n)-{\cal R}_{\rho_g}({\cal N}_{0,1})\big|\\
    & \le & \sup_{g\in{\cal G}_\rho}\int_{-\infty}^\infty\big|g(F_{\mathfrak{m}_n}(x))-g(F_{{\cal N}_{0,1}}(x)\big|\,dx\\
    & \le & \int_{-\infty}^\infty L|F_{\mathfrak{m}_n}(x)-F_{{\cal N}_{0,1}}(x)|^\beta\,dx\\
    & \le & C\,d_{\phi_\lambda}(\mathfrak{m}_n,{\cal N}_{0,1})^\beta
\end{eqnarray*}
for the constant $C:=L\int_{-\infty}^\infty 1/\phi_\lambda(x)^\beta dx$ (which is finite due to the assumption $\lambda \beta>1$). That is, condition (d) is satisfied too.

It remains to show (\ref{prop on Orlicz rm - PROOF - 20}), for which we will adapt the arguments of Section 4.3 in \cite{KraetschmerZaehle2011}. Let $0\le t<t'<1$. Since the underlying probability space was assumed to be atomless, we may pick a measurable decomposition $A_1\cup A_2\cup A_3$ of the probability domain such that ${\rm P}[A_1]=1-t'$, ${\rm P}[A_2]=t'-t$ and ${\rm P}[A_3]=t$, where ${\rm P}$ refers to the corresponding probability measure. Define random variables $B_{1-t'}:=\eins_{A_1}$, $B_{1-t}:=\eins_{A_1\cup A_2}$ and $B_{t'-t}:=\eins_{A_2}$, and note that they are distributed according to the Bernoulli distribution with parameters $1-t'$, $1-t$ and $t'-t$, respectively. Moreover we clearly have $B_{1-t}=B_{1-t'}+B_{t'-t}$. By the subadditivity of $\rho_g$ we can conclude $\rho(B_{1-t})\le\rho(B_{1-t'})+\rho(B_{t'-t})$, and so
\begin{eqnarray*}
    g(t') - g(t) & = & 1-\rho_g(B_{1-t'}) - (1-\rho_g(B_{1-t}))\\
    & \le & \rho_g(B_{t'-t})\\
    & \le & \rho(B_{t'-t})\\
    & \le & \sup_{u\in(0,1]}\frac{\rho(B_{u})}{u^{\beta}}\,(t'-t)^{\beta}\\
    & \le & \sup_{u\in(0,1]}\frac{1-g_\rho(1-u)}{u^{\beta}}\,(t'-t)^{\beta}\\
    & \le & \sup_{v\in[0,1)}\frac{1-g_\rho(v)}{(1-v)^{\beta}}\,(t'-t)^{\beta}
\end{eqnarray*}
for every $g\in{\cal G}_\rho$, where the second ``$\le$'' is ensured by (\ref{prop on Orlicz rm - PROOF - 10}). By (\ref{prop on Orlicz rm - eq}) the constant $\sup_{v\in[0,1)}\frac{1-g_\rho(B_{v})}{(1-v)^{\beta}}$ is finite. Thus, since every $g\in{\cal G}_\rho$ is also continuous at $1$, condition (\ref{prop on Orlicz rm - eq}) indeed implies (\ref{prop on Orlicz rm - PROOF - 20}).
\end{proof}

It is worth mentioning that if $\rho$ is a distortion risk measure with distortion function $g$, then $g_\rho=g$ and condition (\ref{prop on Orlicz rm - eq}) boils down to condition (\ref{prop on drm - eq}). Here are two examples for law-invariant coherent risk measures on Orlicz hearts that are {\em not} distortion risk measures:

\begin{examplenorm}\label{example osmrm}
Given $a\in(0,1]$ and $p\in[1,\infty)$, the {\em one-sided $p$th moment risk measure} is the map $\rho:L^p\rightarrow\R$ defined by
\begin{eqnarray*}
     \rho(X)\,:=\,{\rm E}[X]+a{\rm E}[((X-{\rm E}[X])^{+})^p]^{1/p},
\end{eqnarray*}
where ${\rm E}$ refers to the expectation w.r.t.\ the probability measure of the basic probability space. The map $\rho$ is clearly law-invariant and can easily be shown to be a coherent risk measure. But by Lemma A.5 in \cite{KraetschmerZaehle2011} it is not a distortion risk measure.

The function $g_\rho$ defined in Lemma \ref{prop on Orlicz rm} is given by $g_{\rho}(t)=t-at(1-t)^{1/p}$, and thus
$$
    1-g_{\rho}(t)\,=\,1-t+at(1-t)^{1/p}\,\le\,(1+a)(1-t)^{1/p}\quad\mbox{ for all }t\in[0,1].
$$
Therefore condition (\ref{prop on Orlicz rm - eq}) is satisfied for $\beta:=1/p$.
{\hspace*{\fill}$\Diamond$\par\bigskip}
\end{examplenorm}

\begin{examplenorm}\label{example rmboe}
In \cite{Bellinietal2014} it has been pointed out that expectiles may be viewed as law-invariant coherent risk measures. The {\it expectiles-based risk measure} associated with $\alpha\in [1/2,1)$ is the map $\rho:L^2\rightarrow\R$ defined by
$$
     \rho(X):=\,{\rm argmin}_{m\in\R}\,\big\{\alpha\|(X-m)^{+}\|_{2}^{2}+(1-\alpha)\|(m-X)^{+}\|_{2}^{2}\big\}.
$$
It follows from Theorem 8 in \cite{Delbaen2013} that $\rho$ is not a distortion risk measure unless $\alpha=1/2$.

The function $g_\rho$ defined in Lemma \ref{prop on Orlicz rm} is given by  $g_{\rho}(t)=(1-\alpha)t/(1-\alpha+(1-t)(2\alpha-1))$, and thus
$$
    1-g_{\rho}(t)\,=\,\alpha(1-t)/(1-\alpha+(1-t)(2\alpha-1)))\,\le\,(\alpha/(1-\alpha))(1-t)\quad\mbox{ for all }t\in[0,1].
$$
Therefore condition (\ref{prop on Orlicz rm - eq}) is satisfied for $\beta:=1$.
{\hspace*{\fill}$\Diamond$\par\bigskip}
\end{examplenorm}

%%%%%%%%%%%%%%%%%%%%%%%%%%%%%%%%%%%%%%%%%%%%%%%%%%%%%%%%%%%%%%%%
%%%%%%%%%%%%%%%%%%%%%%%%%%%%%%%%%%%%%%%%%%%%%%%%%%%%%%%%%%%%%%%%
%%%%%%%%%%%%%%%%%%%%%%%%%%%%%%%%%%%%%%%%%%%%%%%%%%%%%%%%%%%%%%%%
%%%%%%%%%%%%%%%%%%%%%%%%%%%%%%%%%%%%%%%%%%%%%%%%%%%%%%%%%%%%%%%%
%%%%%%%%%%%%%%%%%%%%%%%%%%%%%%%%%%%%%%%%%%%%%%%%%%%%%%%%%%%%%%%%
%%%%%%%%%%%%%%%%%%%%%%%%%%%%%%%%%%%%%%%%%%%%%%%%%%%%%%%%%%%%%%%%

\section{Numerical examples}\label{Numerical examples}

In this section we present some numerical examples to illustrate the results of Section \ref{main results}. Our results show that both the estimated normal approximation and the empirical plug-in estimator lead to reasonable estimators for the
premium of an individual risk within a homogeneous insurance collective. Our results also show that these two estimators are asymptotically equivalent. Nevertheless for small to moderate collective sizes $n$ the goodness of the estimators can
vary from case to case. For example, in the case where $\rho$ is the Value at Risk at level $\alpha$ the results of the Theorems \ref{main theorem} and \ref{main theorem - EPiE}  show that for both estimators the estimation error converges almost surely to zero at rate (nearly) $1/2$ when $\ex[|Y_1|^\lambda]<\infty$ for some $\lambda>2$ (where $Y_1$ refers to any $\mu$-distributed random variable). On the other hand, the latter condition does not exclude that $\ex[|Y_1|^{2+\varepsilon}]=\infty$ for some small $\varepsilon>0$. In this case the total claim
distribution can be essentially skewed to the right when the number of individual risks $n$ is small to moderate; cf.\ Figure \ref{fig 1}. So one would expect that especially for heavy-tailed $\mu$ and small to moderate $n$ the estimators perform only moderately well. One would also expect that for heavy-tailed $\mu$ (and even for medium-tailed $\mu$) and small to moderate $n$ the empirical plug-in estimator should outperform the estimated normal approximation. Our goal in this section is to provide empirical evidence for our conjectures.

To this end let us consider a sequence $(Y_i)$ of i.i.d.\ nonnegative random variables on a common probability space with distribution
$$
    \mu = (1-p)\,\delta_{0}+p\,{\rm P}_{a,b}
$$

for some $p\in(0,1)$, where ${\rm P}_{a,b}$ is the Pareto distribution with parameters $a>2$ and $b>0$. The Pareto distribution ${\rm P}_{a,b}$ is determined by the Lebesgue density
$$
    f_{a,b}(x):=ab^{-1}\big(b^{-1}x+1\big)^{-(a+1)}\,\eins_{(0,\infty)}(x), %,\qquad x\in\R.
$$
and the assumption $a>2$ ensures that $\ex[|Y_1|^\lambda]<\infty$ for all $\lambda\in(2,a)$. We regard $Y_1,\ldots,Y_n$ as a homogeneous insurance collective of size $n$, the number $p$ as the probability for the event of a strictly positive
individual claim amount, and ${\rm P}_{a,b}$ as the individual claim distribution conditioned on this event. Note that in our example the mean $m$ and the variance $s^2$ of $\mu$ are given by
\begin{equation}\label{tatsaechliche parameter 1}
    m=\frac{p\,b}{a-1}\quad\mbox{ and }\quad s^2=\,\frac{2b^2p}{(a-1)(a-2)}-\frac{b^2p^2}{(a-1)^2}\,.
\end{equation}

\begin{figure}[htp]
\begin{center}
\includegraphics[width=5cm]{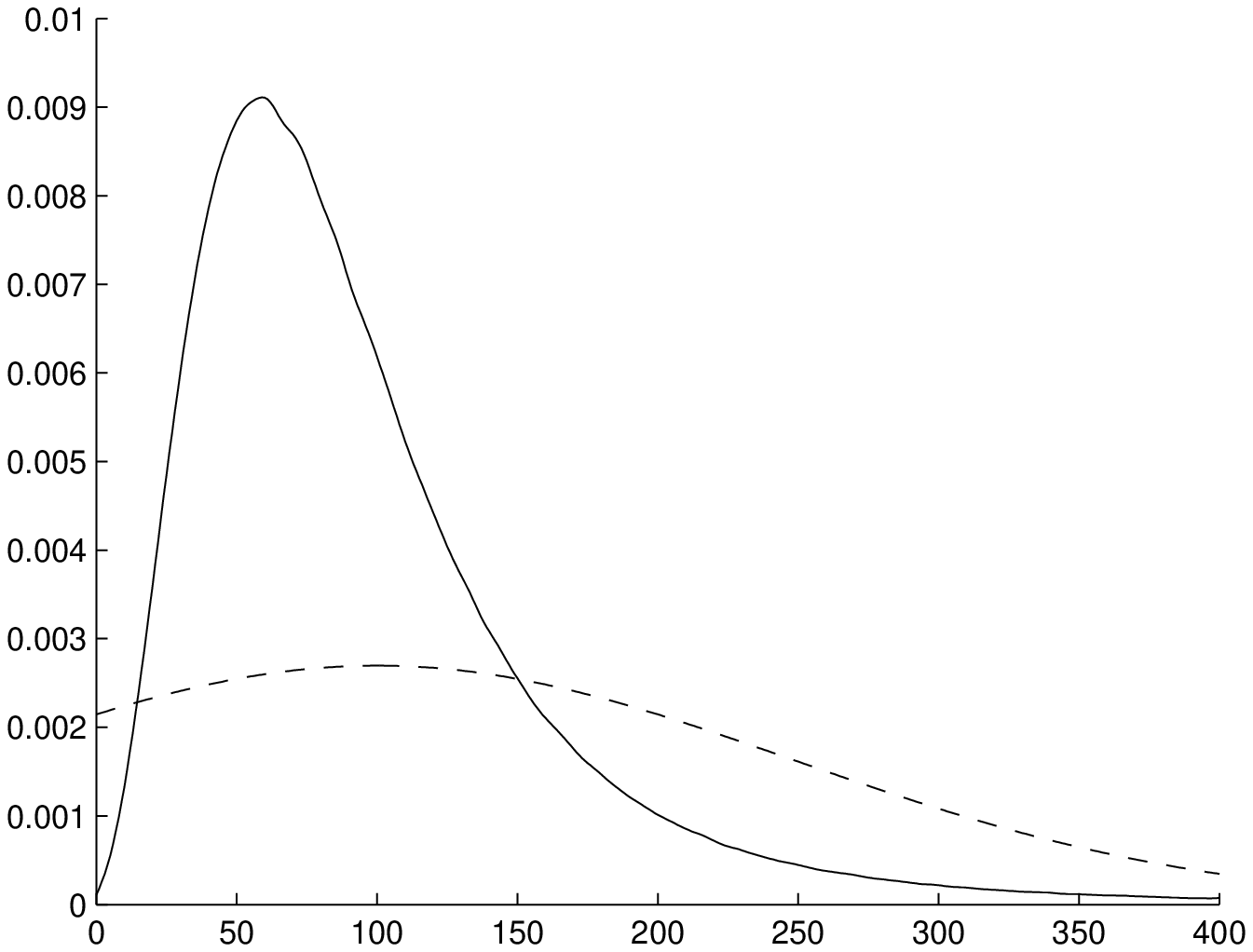}\includegraphics[width=5cm]{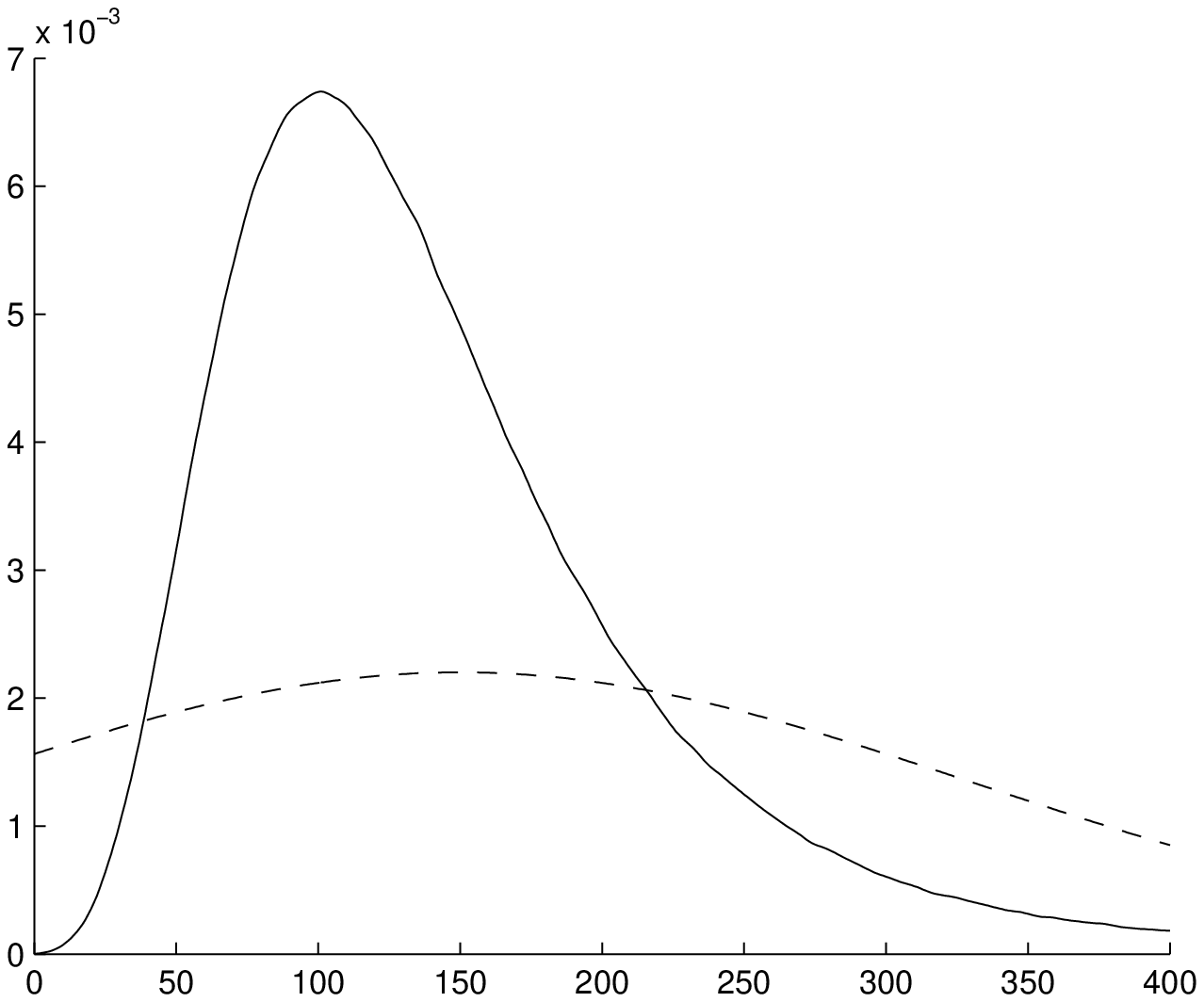}\includegraphics[width=5cm]{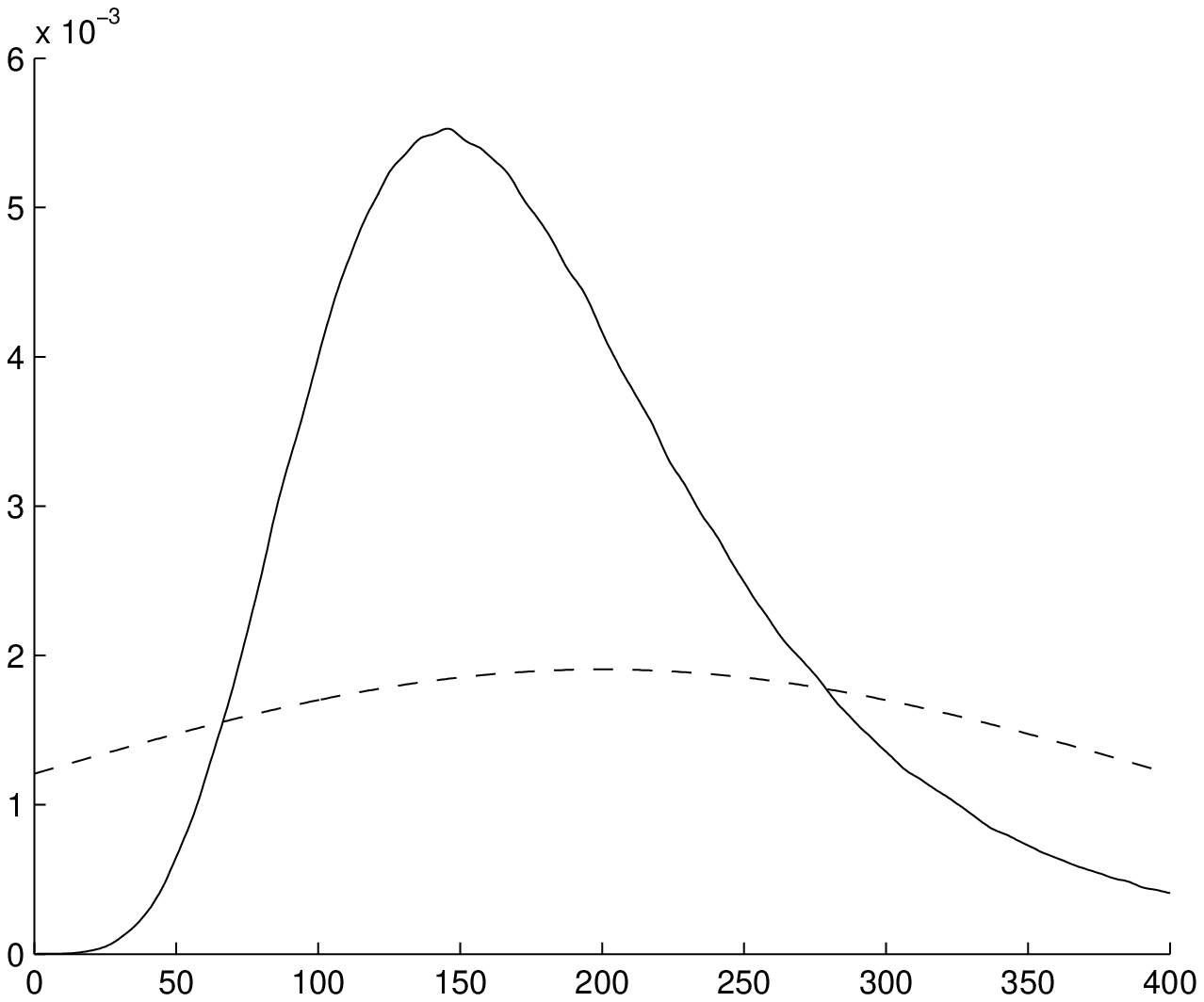}

\includegraphics[width=5cm]{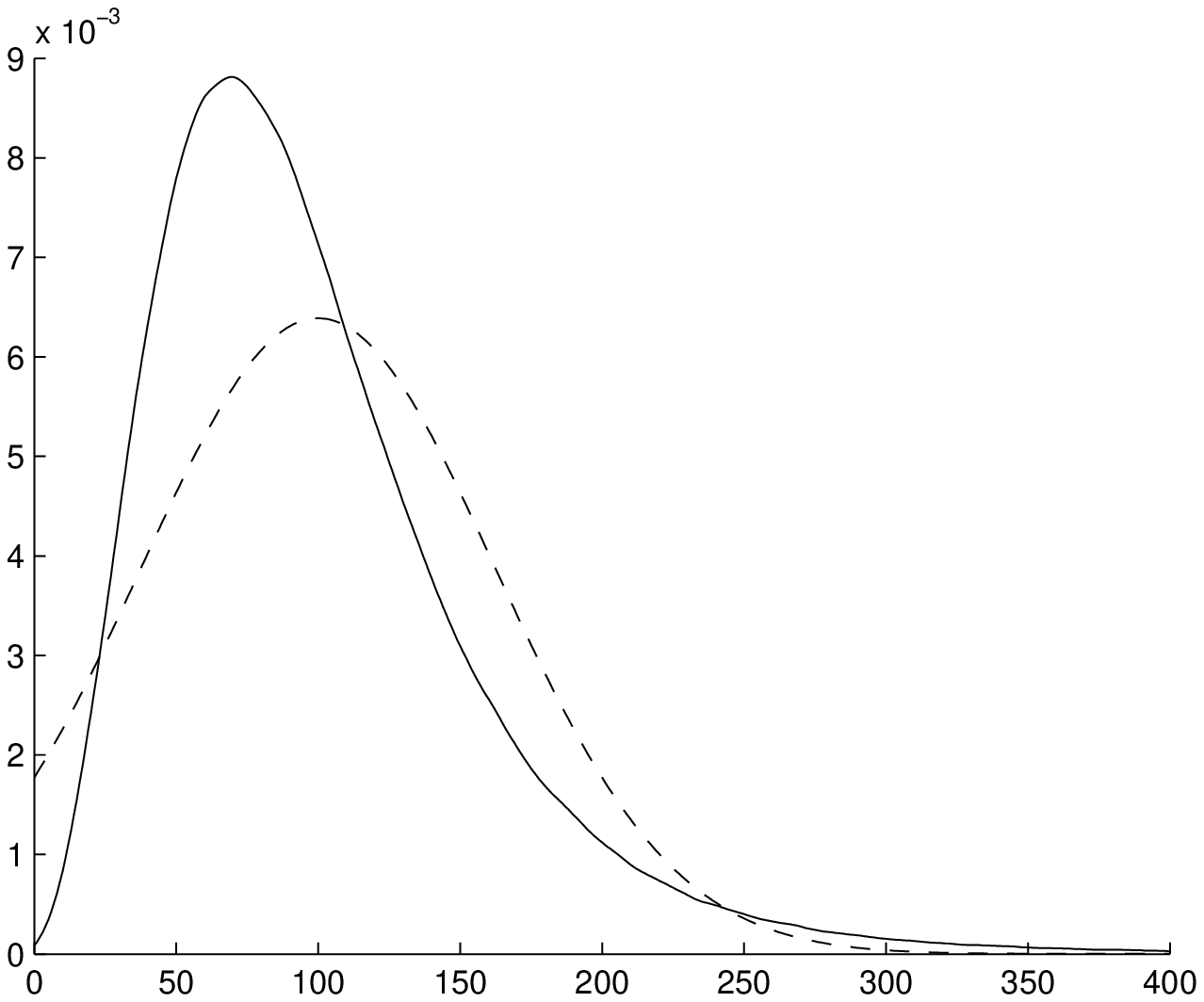}\includegraphics[width=5cm]{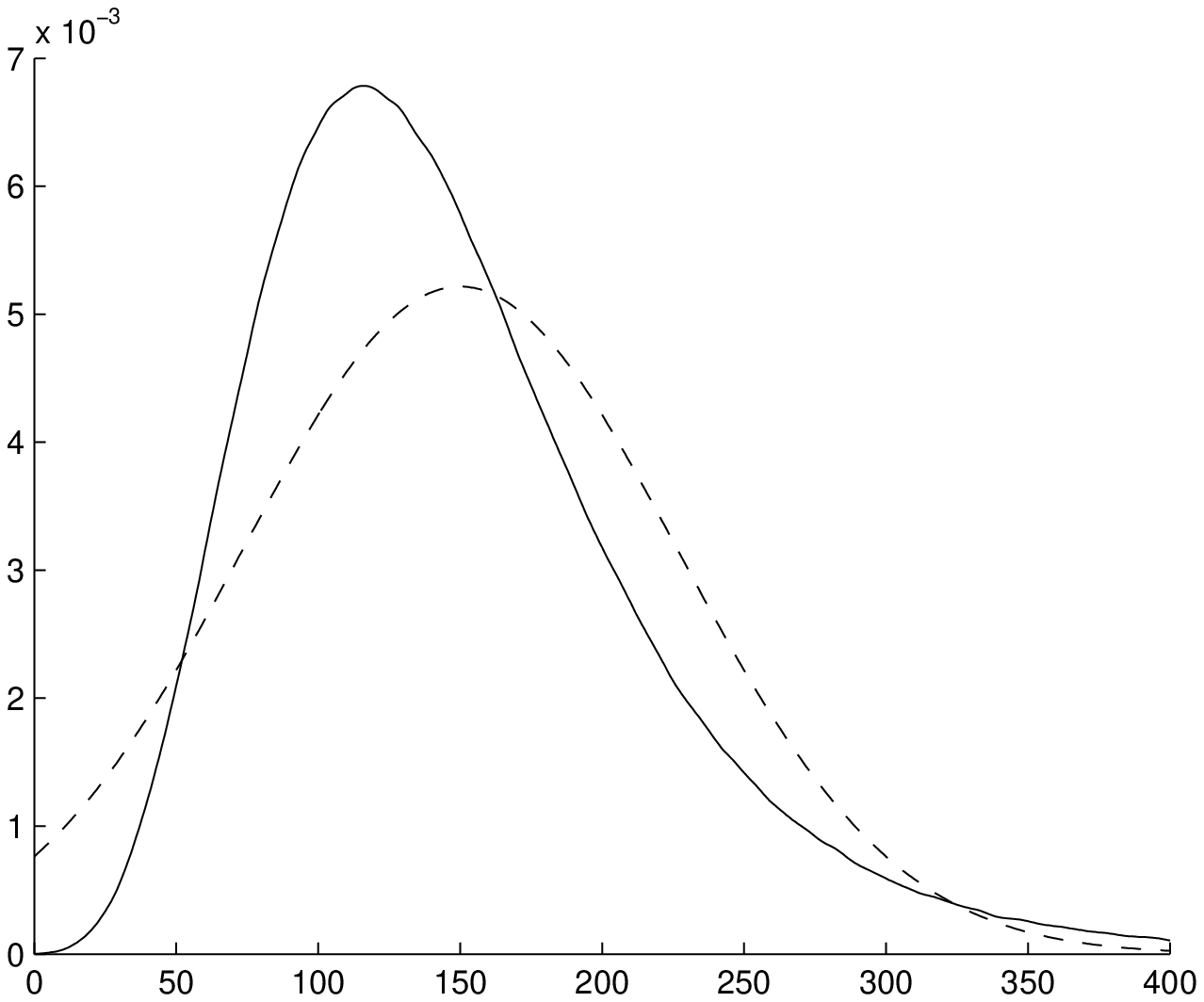}\includegraphics[width=5cm]{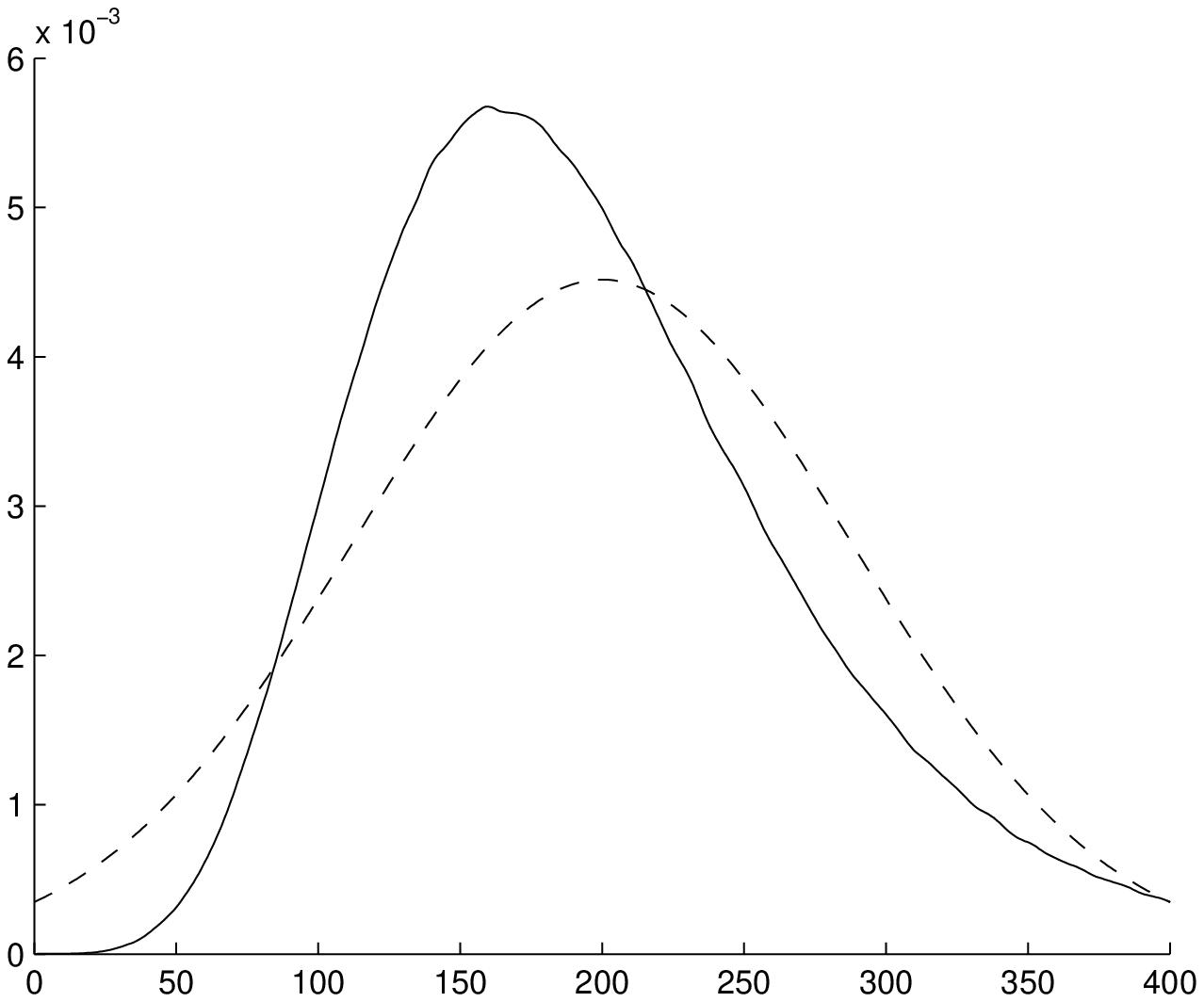}

\includegraphics[width=5cm]{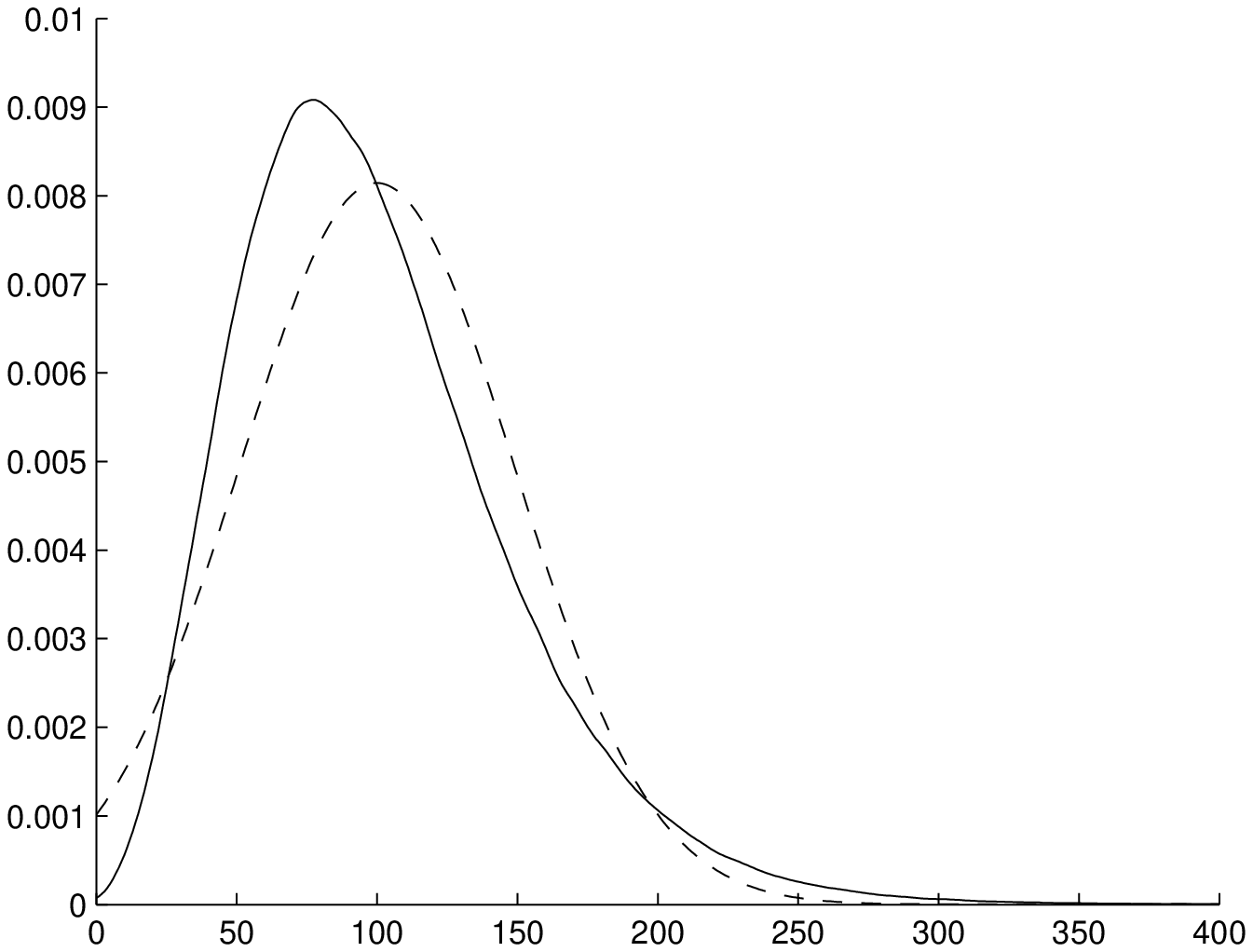}\includegraphics[width=5cm]{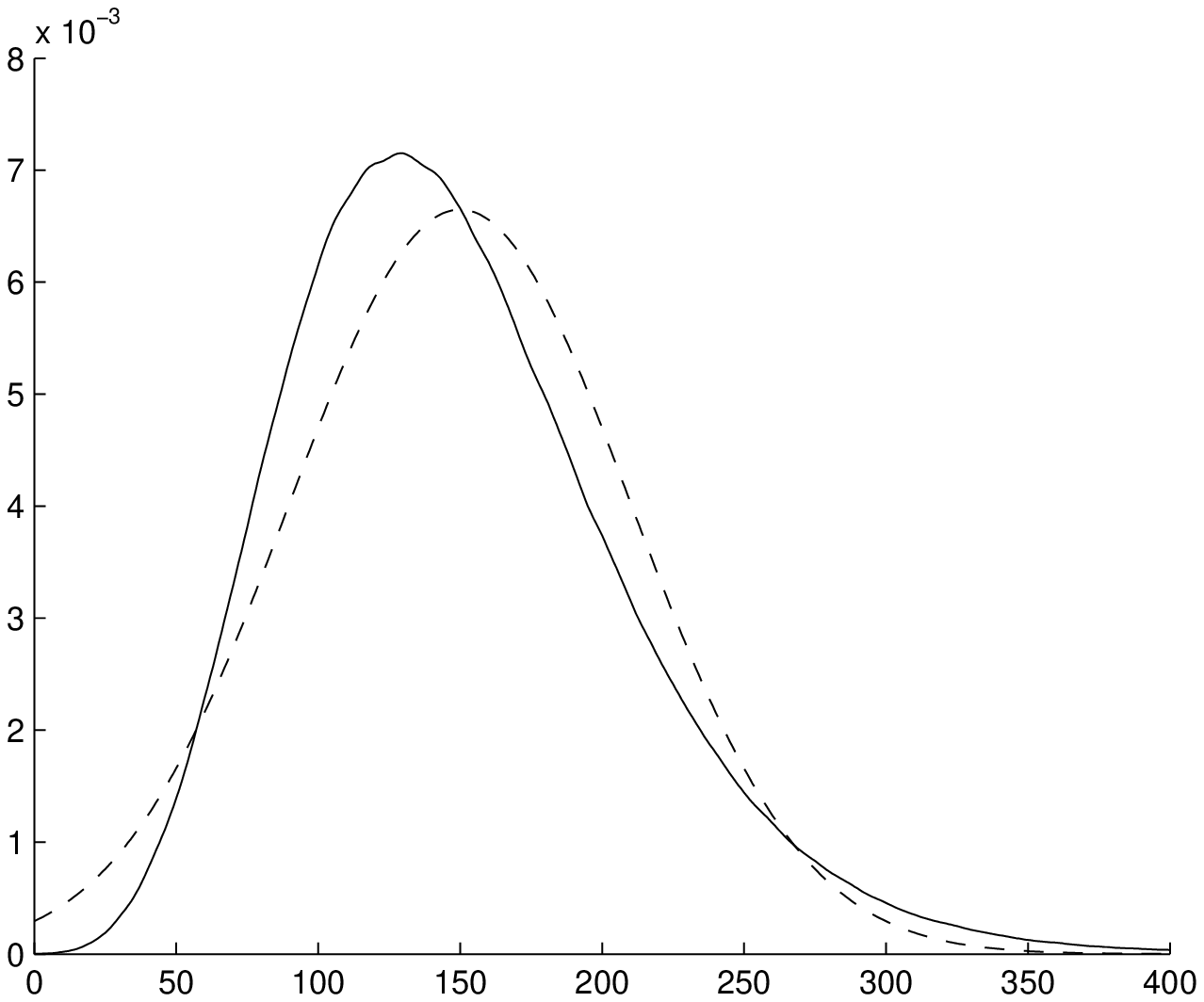}\includegraphics[width=5cm]{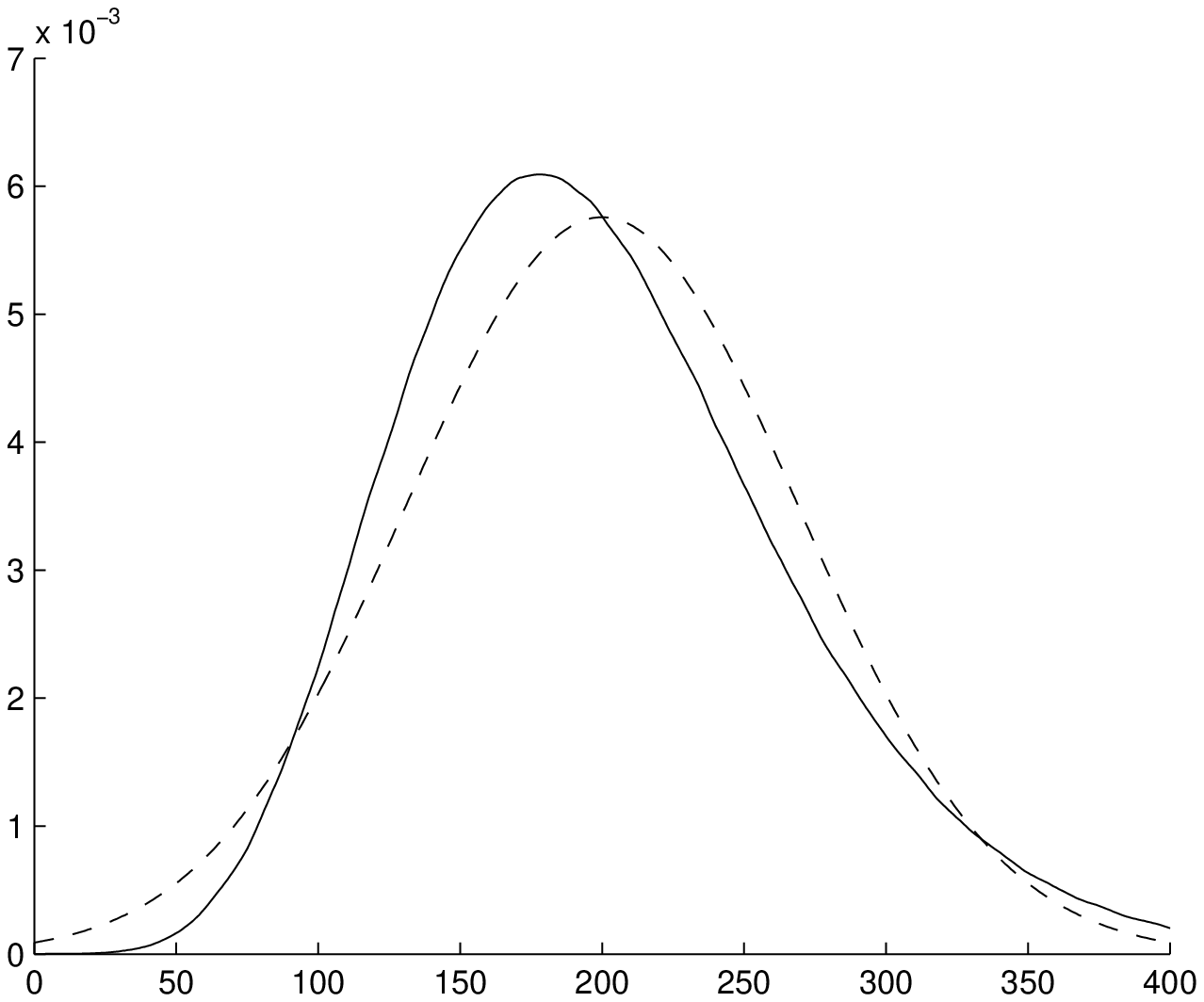}

\includegraphics[width=5cm]{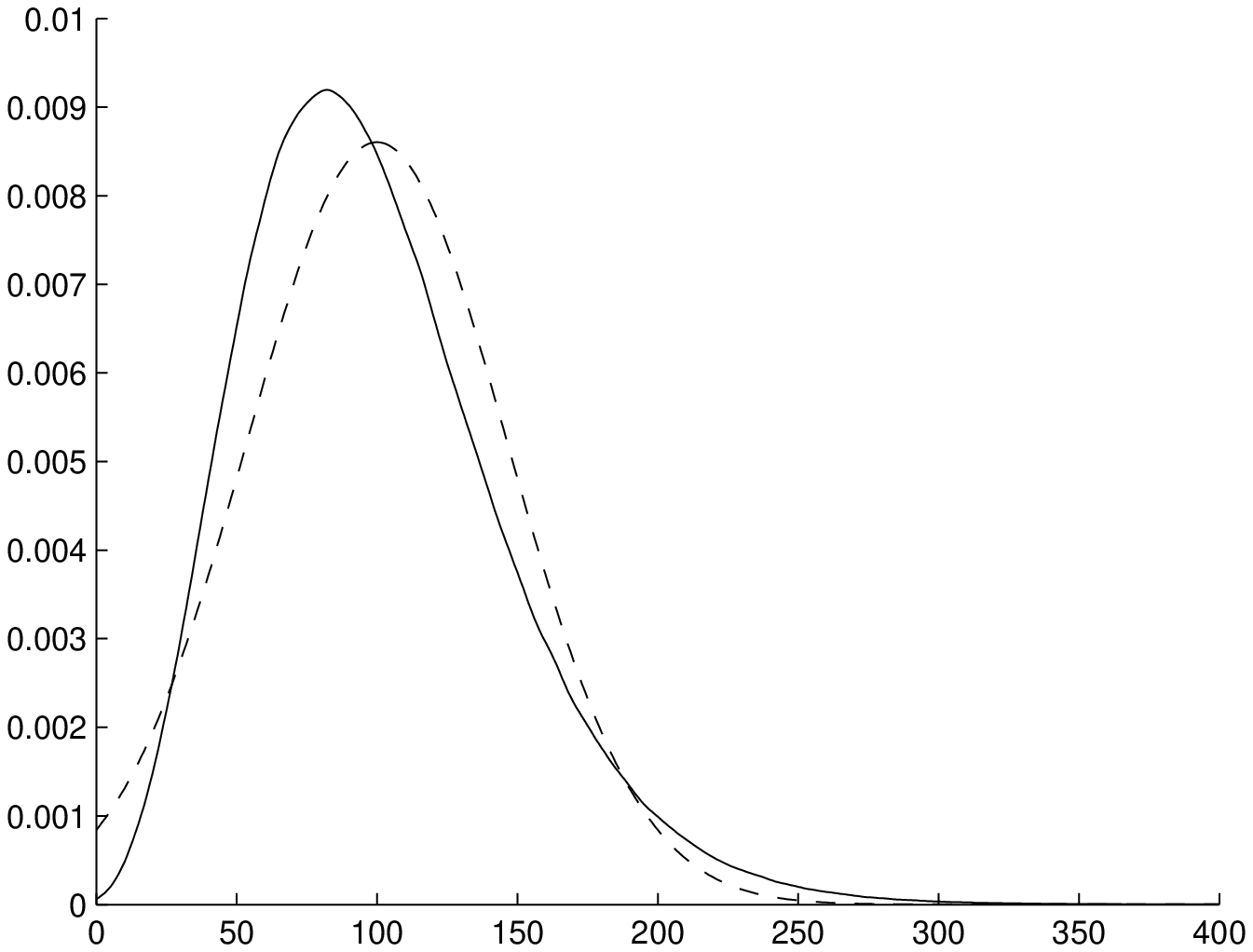}\includegraphics[width=5cm]{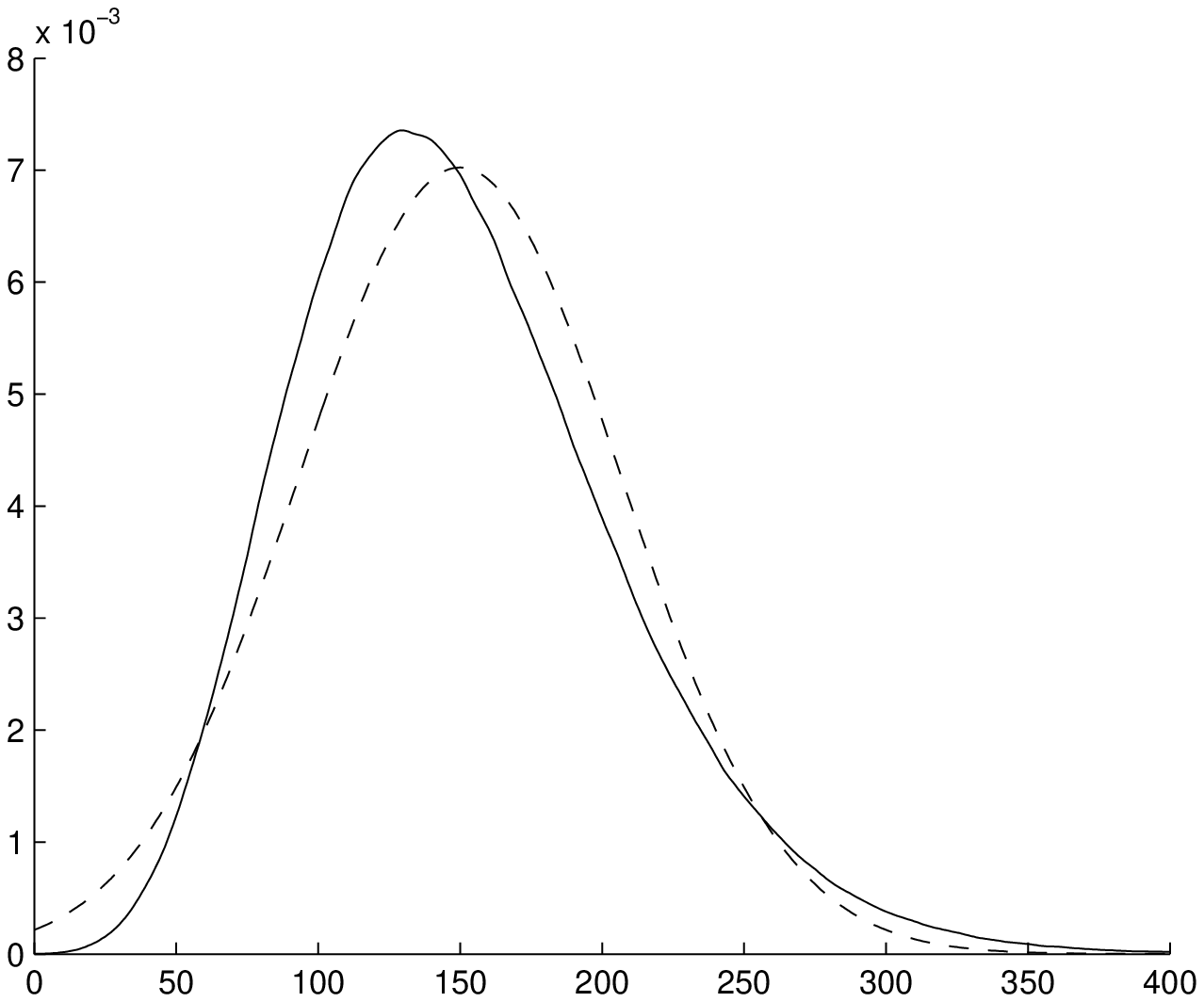}\includegraphics[width=5cm]{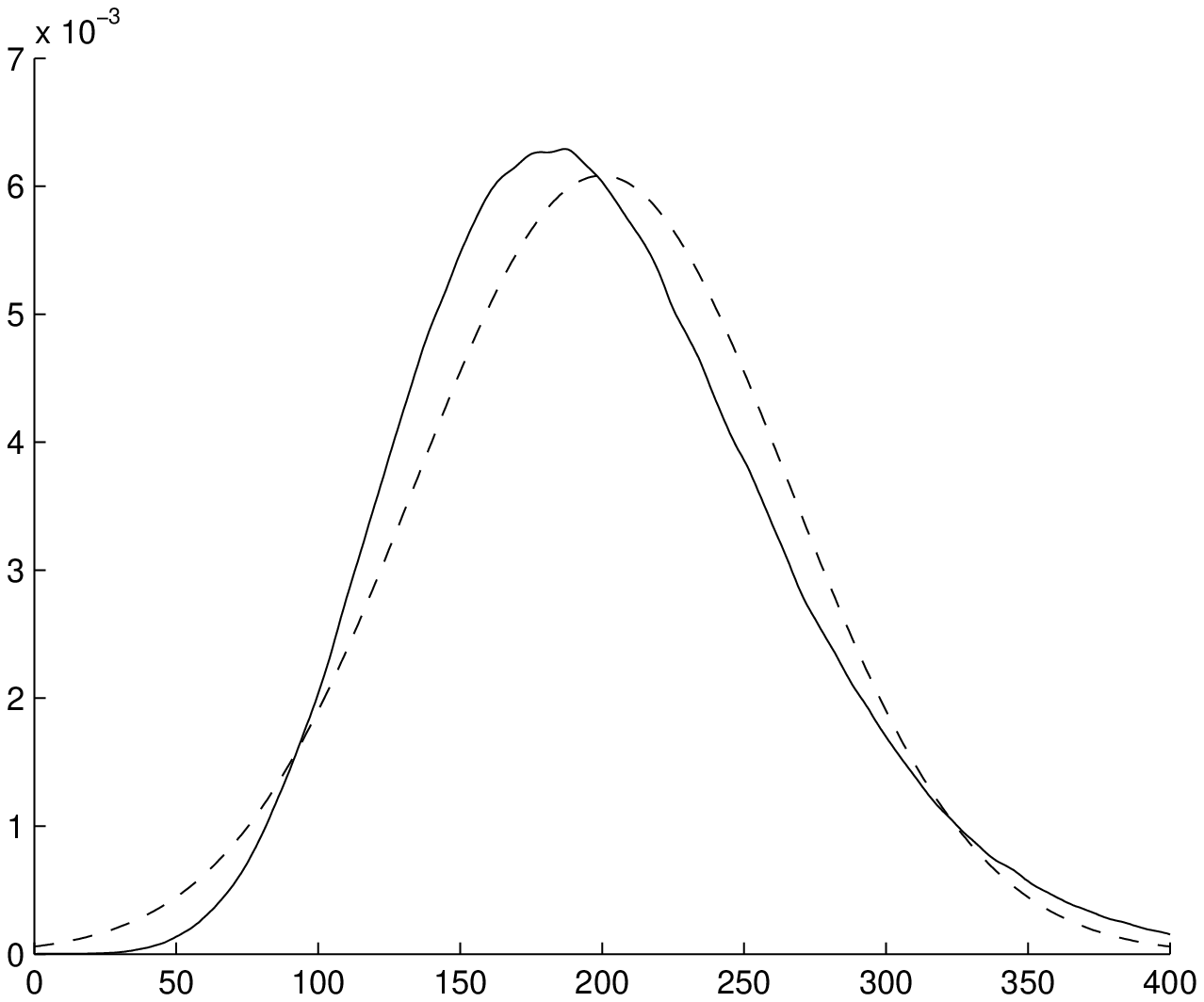}
\end{center}
\caption{The continuous line shows the $n$-fold convolution $\mu^{*n}$ of $\mu=(1-p)\delta_0+p{\rm P}_{a,b}$ for $p=0.1$ and the Pareto distribution ${\rm P}_{a,b}$ with parameter $a=2.1$ in the first line, $a=3$ in the second line, $a=6$ in the third line and $a=10$ in the fourth line and collective sizes $n=100$ in the first row, $n=150$ in the second row and $n=200$ in the third row. The dashed line shows the density of the respective normal distribution in each case.}
\label{fig 1}
\end{figure}

In the first part of this section, we estimate the total claim distribution $\mu^{*n}$, i.e.\ the distribution of $\sum_{i=1}^nY_i$, by means of the empirical distribution based on a Monte-Carlo simulation. The plots in Figure \ref{fig 1} were
derived from a simulation with 100.000 Monte-Carlo paths. We set $p=0.1$ and chose the parameters $a$ and $b$ in such a way that the expected value of a single claim was normalised to $1$. Each line shows the same set of parameters and each row shows the same collective size, starting
with $n=100$ on the left, $n=150$ in the middle and $n=200$ on the right. The first line shows the results for $a=2.1$ and $b=11$, the second line shows $a=3$ and $b=20$, the third
line shows $a=6$ and $b=50$ and the fourth line shows $a=10$ and $b=90$. In each plot the continuous line represents the estimator for $\mu^{*n}$ and the dashed line the probability
density of the normal distribution ${\cal N}_{nm,\,ns^2}$ with $m$ and $s^2$ determined through (\ref{tatsaechliche parameter 1}). We emphasize that $\mu^{*n}$ has in fact point mass in zero. But the point mass is equal to $(1-p)^n$ and
therefore extremely small. This is why the point mass of the empirical estimator is not visible in the plots.

One can see that the empirical total claim distributions in the first line of Figure \ref{fig 1} are strongly skewed to the right even for larger collective sizes. The density of the normal distribution is very flat and has much mass on the negative semiaxis. The reason for this shape is the high variance $s^2$, which increases rapidly as $a$ gets closer to $2$. In the case of $a=2.1$ and $b=11$ this rate is close to zero, saying that large collective sizes are needed to provide a suitable estimator.

In the second line of Figure \ref{fig 1} for $a=3$ and $b=20$ the empirical total claim distributions are still strongly skewed to the right. One can see that the normal
approximation still does not resemble the empirical distribution. The deviation decreases visibly with increasing collective size due to
the higher rate of convergence in the Berry--Esséen theorem. Compared to the first line with $a=2.1$ and $b=11$ the quality of the
normal approximation was increased in the second line with $a=3$ and $b=20$, which can be explained by the increasing rate of convergence in the Berry--Esséen theorem. For
$\lambda\in(2,3]$ the convergence rate to the normal distribution is strictly increasing in $\lambda$. For $\lambda>3$ the convergence
rate can not be improved any more.

In the third and fourth line of Figure \ref{fig 1} for $a=6$ and $b=50$ and $a=10$ and $b=90$ the normal approximation provides a good approximation even for small collective sizes. The empirical total claim distributions are in both cases
almost symmetric and the approximation leads to a good fit of both curves. The third moment of $X_1$ exists in both cases and due to the Berry--Esséen theorem the deviation of $\mu^{*n}$ from the normal distribution converges to zero with rate
$1/2$. We can see that there is no remarkable improval
in the convergence rate once the existence of the third moment is guaranteed.

In the second part of this section we compare the estimated normal approximation with the empirical plug-in estimator where the role of the risk measure $\rho$ is played by the
Value at risk at level $\alpha=0.99$. To save computing time we discretized the Pareto distribution ${\rm P}_{a,b}$ on the equidistant grid $10\N_0=\{0,10,20,\ldots\}$.
The plots in Figure \ref{fig 2} were derived by a Monte-Carlo method using 100 Monte-Carlo paths in each simulation. Once again we chose $p=0.1$. In order to compare the
estimators we first calculated the exact Value at Risks at level $0.99$ of $\mu^{*n}$ (in fact we estimated it by means of a Monte-Carlo simulation based on
100.000 runs) in dependence on the collective size $n$. In each plot in Figure \ref{fig 2} the dotdashed line represents the relative Value at Risk
${\cal R}_\rho(\mu^{*n})/n$, which we take as a reference to illustrate the biases of the estimators. The dashed line shows the estimated normal approximation $\mathcal{R}_{\rho}({\cal N}_{n\widehat{m}_n,\,n\widehat{s}_n^2})/n$ for the Value at Risk relative to $n$. The continuous line shows the empirical plug-in estimator
$\mathcal{R}_{\rho}(\widehat\mu_{n}^{\,*n})/n$ for the Value at Risk relative to $n$.

\begin{figure}[htb]
\begin{center}
\includegraphics[width=7.8cm]{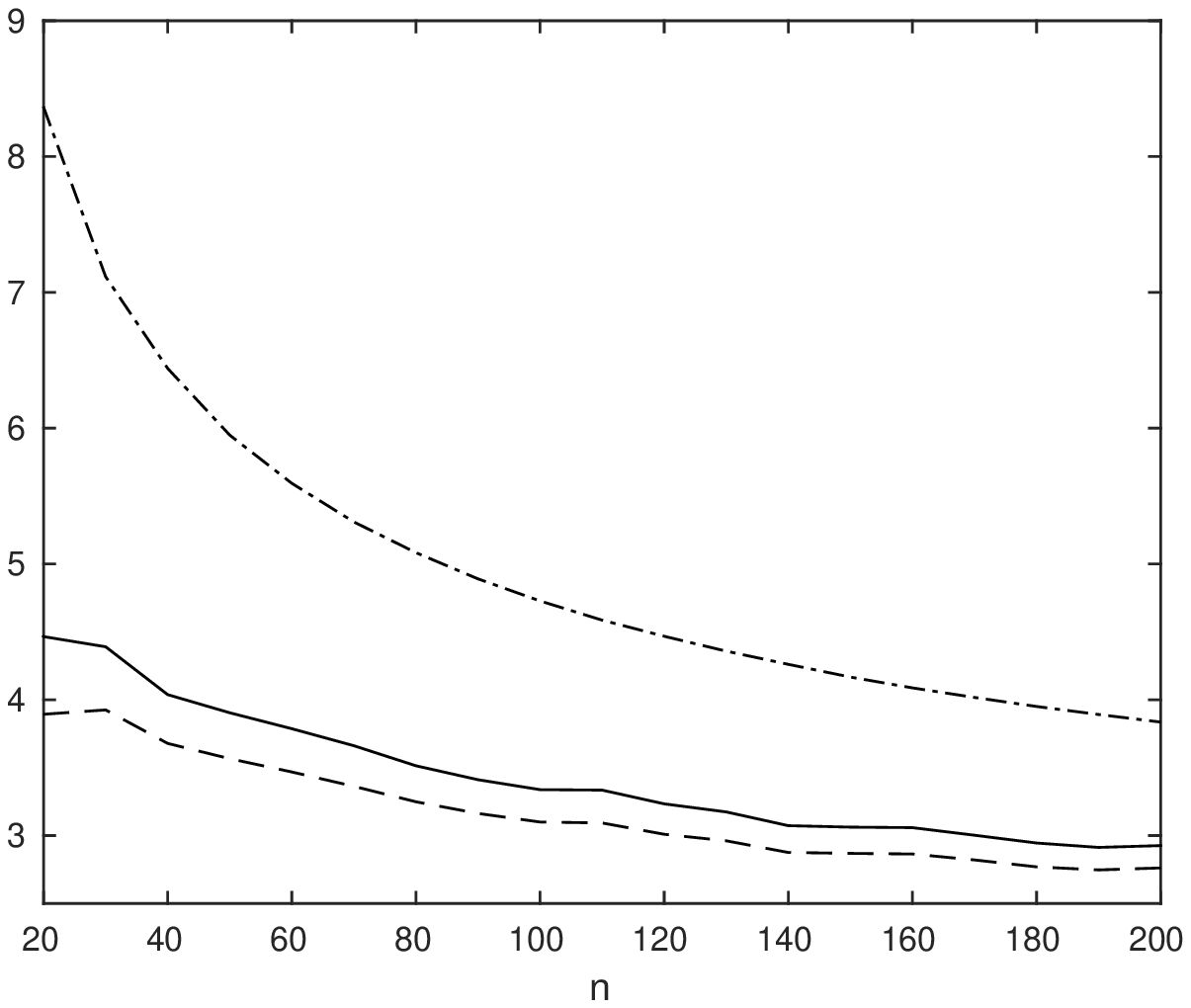}\includegraphics[width=7.8cm]{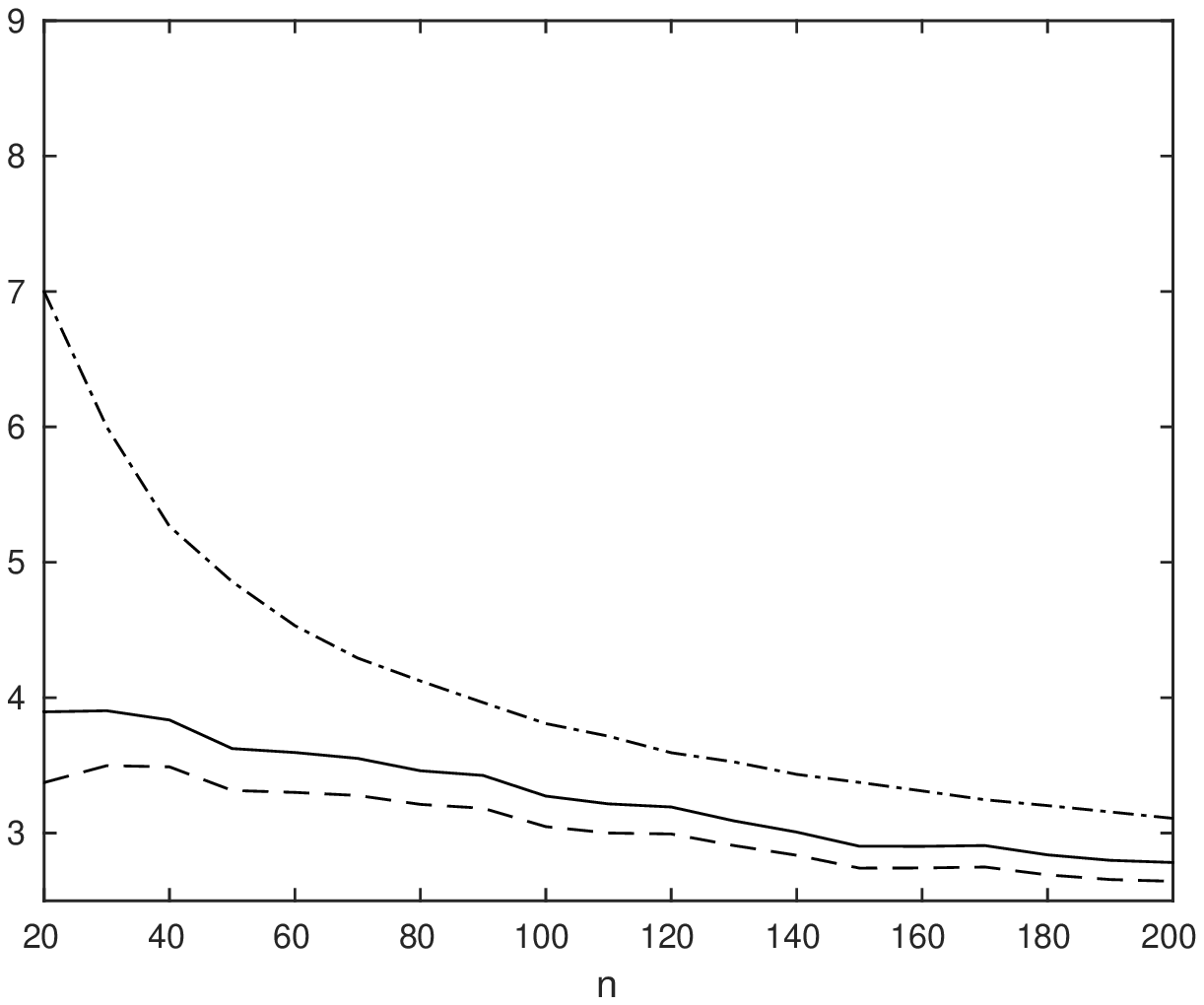}

\includegraphics[width=7.8cm]{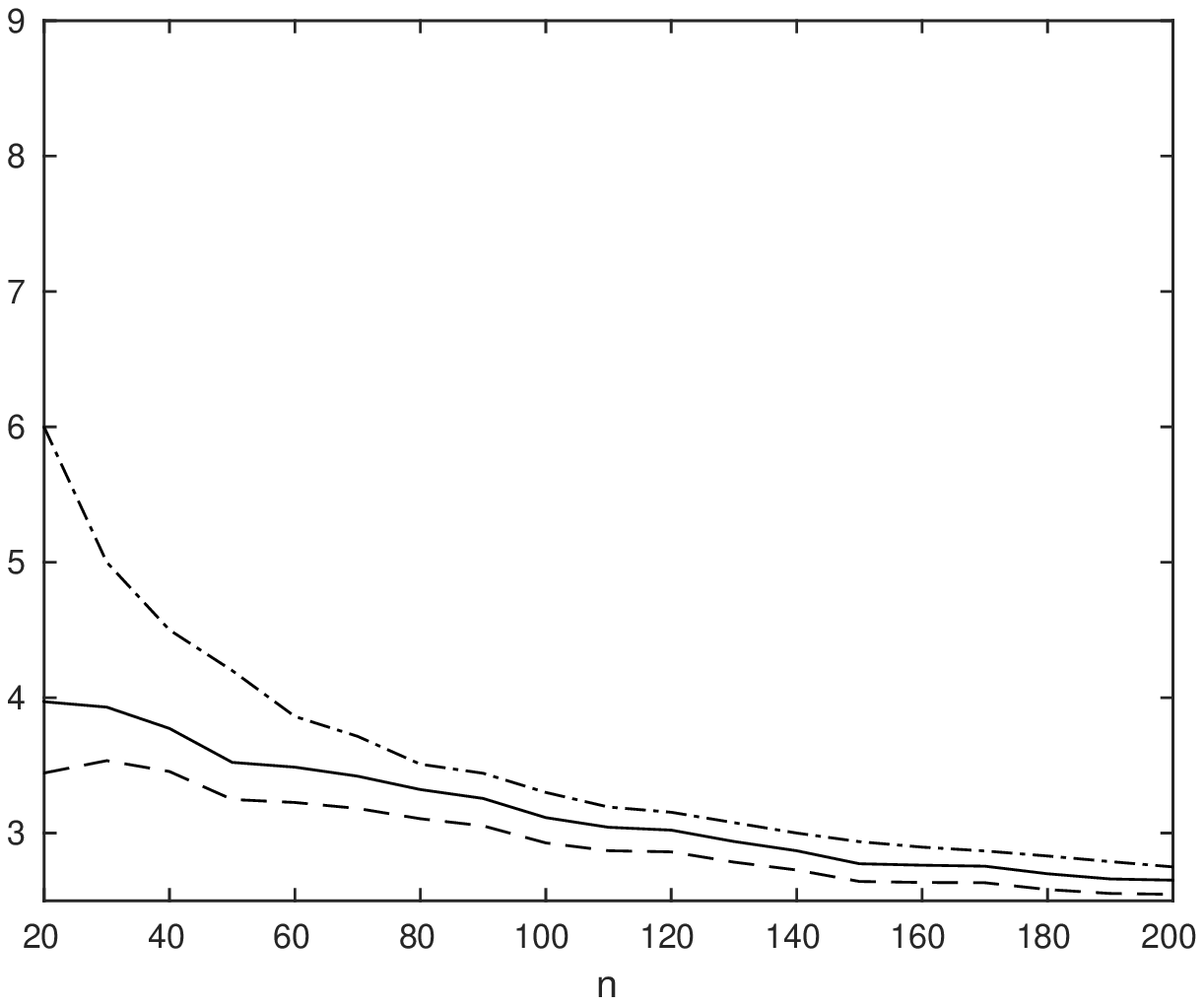}\includegraphics[width=7.8cm]{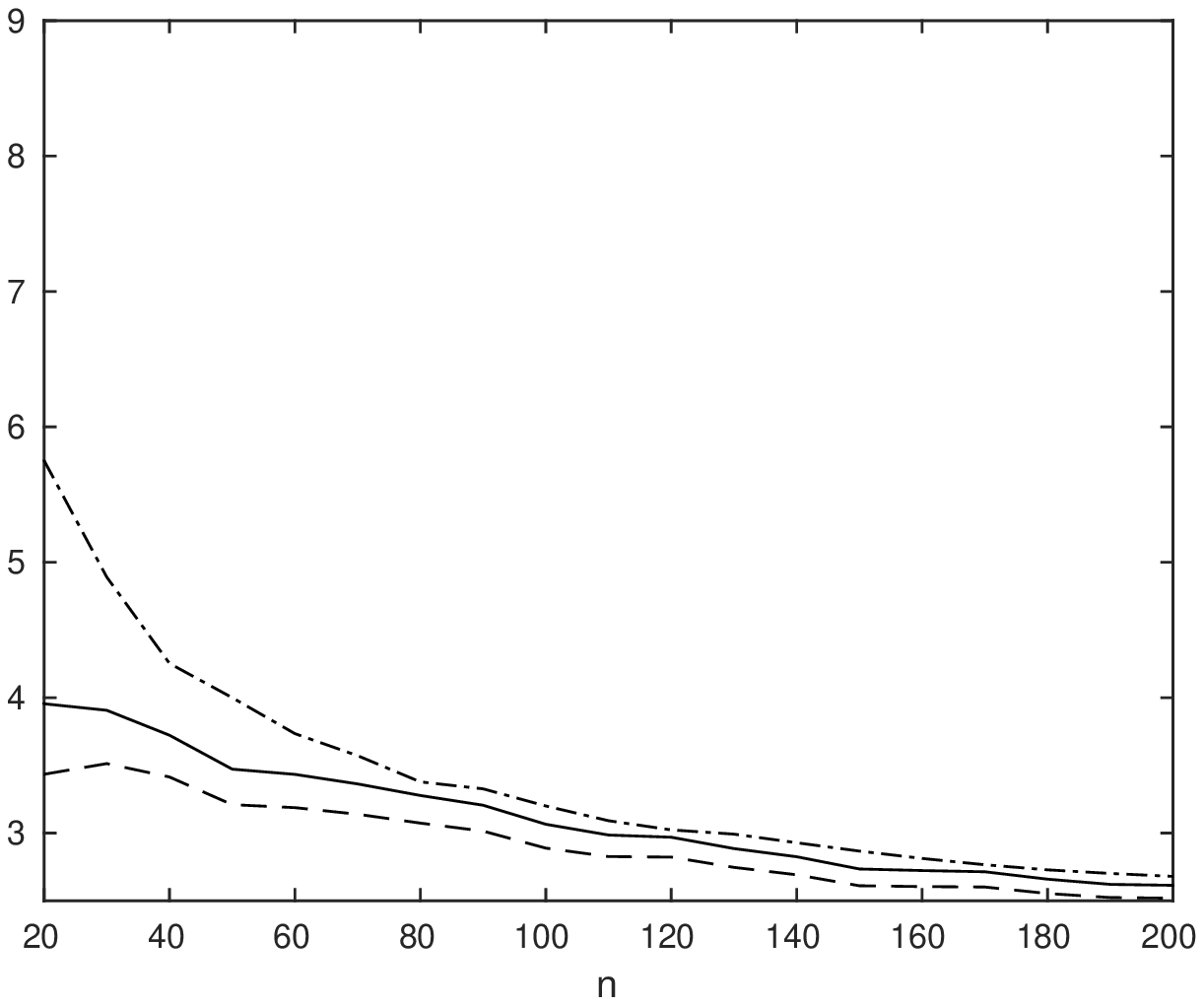}
\end{center}
\caption{${\cal R}_\rho(\mu^{*n})/n$ (dotdashed line) as well as the average of $100$ Monte-Carlo paths of respectively
$\mathcal{R}_{\rho}({\cal N}_{n\widehat{m}_n,\,n\widehat{s}_n^2})/n$ (dashed line) and $\mathcal{R}_\rho(\widehat{\mu}_n^{\,*n})/n$
(continuous line) for $\rho=\vatr_{0.99}$ in dependence on the collective size $n$, showing $a=2.1$ on the left hand side and $a=3$ on the right hand side of the first line and $a=6$ on the left hand side and $a=10$ on the right hand side of the second line.
}
\label{fig 2}
\end{figure}

The first line shows the relative Value at Risks for the parameters $a=2.1$ and $b=11$ on the left and $a=3$ and $b=20$ on the right hand side. In the second line we have $a=6$ and $b=50$ on the left and $a=10$ and $b=90$ on the right hand side. Once again the parameters were chosen such that the expected value of a single claim was normalised to $1$.

For $a=2.1$ we can see that both estimators show a large negative bias. The slow convergence in the Berry--Esséen theorem transfers directly to the convergence of the relative Value at risk of the distributions (recall that the Value at Risk fulfills condition (d) of Assumption \ref{basic assumption} for $\beta=1$). Due to this slow convergence the collective size has to be chosen very large to provide a good estimation. What strikes the most is the large bias of the relative empirical plug-in estimator $\mathcal{R}_{\rho}({\mu^{*n}})/n$. The heaviness of the tails causes the empirical distribution $\widehat{\mu}_n$ to converge very slowly
to $\mu^{*n}$. We can see that in the case $a=3$ the bias of both estimators decreases visibly. However in both cases the empirical plug-in estimator yields a better estimation.

The plots for $a=6$ and $a=10$ resemble each other very much. In both cases the existence of the third moment of $X_1$ is guaranteed, yielding the same rate of convergence in the Berry--Esséen theorem. We can see that for small $n$, e.g. $n\le 40$, both estimators show a large bias. However for $n\le 100$ the empirical plug-in estimator provides a better estimation. For $n\ge 100$ the estimated normal approximation could be preferred over the empirical plug-in estimator, because the biases of both estimators are more or less the same and the estimated normal approximation consumes less computing time.

As a conclusion one can say that the estimated normal approximation is not suitable for heavy-tailed (to medium-tailed) distributions whenever small collective sizes are at hand. In this case it is sensible to apply the empirical plug-in estimator, which consumes more
computing time compared to the estimated normal approximation.

%%%%%%%%%%%%%%%%%%%%%%%%%%%%%%%%%%%%%%%%%%%%%%%%%%%%%%%%%%%%%%%%
%%%%%%%%%%%%%%%%%%%%%%%%%%%%%%%%%%%%%%%%%%%%%%%%%%%%%%%%%%%%%%%%
%%%%%%%%%%%%%%%%%%%%%%%%%%%%%%%%%%%%%%%%%%%%%%%%%%%%%%%%%%%%%%%%
%%%%%%%%%%%%%%%%%%%%%%%%%%%%%%%%%%%%%%%%%%%%%%%%%%%%%%%%%%%%%%%%
%%%%%%%%%%%%%%%%%%%%%%%%%%%%%%%%%%%%%%%%%%%%%%%%%%%%%%%%%%%%%%%%
%%%%%%%%%%%%%%%%%%%%%%%%%%%%%%%%%%%%%%%%%%%%%%%%%%%%%%%%%%%%%%%%

\section{Proofs}\label{Proofs}

The proof of Theorems \ref{main theorem} and \ref{main theorem - EPiE} avails the following nonuniform Berry--Esséen inequality (\ref{Michels inequality - eq - 20}). The inequality involves the nonuniform Kolmogorov distance $d_{\phi_\lambda}$, which was introduced in (\ref{weighted kolmogorov distance}).

\begin{theorem}\label{Michels inequality}
Let $(X_i)$ be a sequence of i.i.d.\ random variables on some probability space $(\Omega,{\cal F},\pr)$ such that $\vari[X_1]>0$ and $\ex[|X_1|^{\lambda}]<\infty$ for some $\lambda>2$. For every $n\in\N$, let
$$%\begin{equation}\label{Michels inequality - eq - 10}
    Z_n:=\frac{\sum_{i=1}^n(X_i-\ex[X_1])}{\sqrt{n\vari[X_1]}}\,.
$$%\end{equation}
Then there exists a universal constant $C_{\lambda}\in(0,\infty)$ such that
\begin{equation}\label{Michels inequality - eq - 20}
    d_{\phi_\lambda}(\pr_{Z_n},{\cal N}_{0,1}) \,\le\,C_\lambda\,f(\pr_{X_1})\,n^{-\gamma}\quad\mbox{ for all $n\in\N$}
\end{equation}
with $\gamma:=\min\{1,\lambda-2\}/2$, where
\begin{equation}\label{Michels inequality - eq - 30}
    f(\pr_{X_1})\,:=\,
    \left\{
    \begin{array}{rll}
        \frac{\ex[|X_1-\ex[X_1]|^{\lambda}]}{\vari[X_1]^{\lambda/2}} & , & 2<\lambda\le 3\\
        \max\Big\{\frac{\ex[|X_1-\ex[X_1]|^{3}]}{\vari[X_1]^{3/2}};\frac{\ex[|X_1-\ex[X_1]|^{\lambda}]}{\vari[X_1]^{\lambda/2}}\Big\}& , & \lambda>3
    \end{array}
    \right..
\end{equation}
\end{theorem}

By ``universal constant'' we mean that the constant is independent of $\pr_{X_1}$. Inequality (\ref{Michels inequality - eq - 20}) has been proven by Nagaev \cite{Nagaev1965} and Bikelis \cite{Bikelis1966} for $\lambda=3$ and $\lambda\in(2,3]$, respectively. Meanwhile there exist several estimates for the constant $C_\lambda$ for $\lambda\in(2,3]$; see \cite{NefedovaShevtsova2013} and references cited therein. For $\lambda>3$ the inequality is a direct consequence of Theorem 5.15 in \cite{Petrov1975}.

%%%%%%%%%%%%%%%%%%%%%%%%%%%%%%%%%%%%%%%%%%%%%%%%%%%%%%%%%%%%%%%%
%%%%%%%%%%%%%%%%%%%%%%%%%%%%%%%%%%%%%%%%%%%%%%%%%%%%%%%%%%%%%%%%
%%%%%%%%%%%%%%%%%%%%%%%%%%%%%%%%%%%%%%%%%%%%%%%%%%%%%%%%%%%%%%%%

\subsection{Proof of Theorem \ref{main theorem}}

(i): By part (c) of Assumption \ref{basic assumption} and the representation (\ref{risk measure of N n}) (and its analogue in the case of known parameters), we have
\begin{equation}\label{main theorem - proof - eq 05}
    {\cal R}_\rho({\cal N}_{n\widehat m_{u_n},\,n\widehat s_{u_n}^2})-{\cal R}_\rho({\cal N}_{nm,\,ns^2})\,=\,\sqrt{n}(\widehat s_{u_n}-s){\cal R}_\rho({\cal N}_{0,1})+n(\widehat m_{u_n}-m).
\end{equation}
Since the empirical standard deviation $\widehat s_{u_n}$ converges $\pr$-a.s.\ to the true standard deviation $s$, the claim of part (i) follows through dividing Equation (\ref{main theorem - proof - eq 05}) by $n$.

(ii): Let $S_n$ be a random variable with distribution $\mu^{*n}$, set $Z_n:=(S_n-nm)/(\sqrt{n}s)$, and note that $\law\{\sqrt{n}s Z_n+nm\}=\mu^{*n}$. Write $N_n$ for any random variable distributed according to the normal distribution ${\cal N}_{nm,ns^2}$, and note that $Z:=(N_n-nm)/(\sqrt{n}s)$ is ${\cal N}_{0,1}$-distributed. Due to part (c) of Assumption \ref{basic assumption}, we obtain
\begin{eqnarray}
    {\cal R}_\rho({\cal N}_{nm,\,ns^2})-{\cal R}_\rho(\mu^{*n})
    & = & \rho(\sqrt{n}s Z+nm)-\rho(\sqrt{n}s Z_n+nm)\nonumber\\
    & = & \sqrt{n}s(\rho(Z)-\rho(Z_n))\nonumber\\
    & = & \sqrt{n}s({\cal R}_\rho({\cal N}_{0,1})-{\cal R}_\rho(\mathfrak{m}_n))\label{main theorem - proof - eq 10},
\end{eqnarray}
where $\mathfrak{m}_n$ denotes the law of $Z_n$. The nonuniform Berry--Esséen inequality of Theorem \ref{Michels inequality} shows that there exists a constant
$K_\lambda\in(0,\infty)$ such that $d_{\phi_\lambda}({\cal N}_{0,1},\mathfrak{m}_n)\le K_\lambda n^{-\gamma}$ for all $n\in\N$. %, where $\lambda=2+\delta$.
Along with (\ref{main theorem - proof - eq 10}) and part (d) of Assumption \ref{basic assumption}, this ensures that we can find constants $K,\beta\in(0,\infty)$ such that $n^{-1}|{\cal R}_\rho({\cal N}_{nm,\,ns^2})-{\cal R}_\rho(\mu^{*n})|\le
n^{-1/2}Kd_{\phi_\lambda}({\cal N}_{0,1},\mathfrak{m}_n)^\beta\le CK_\lambda n^{-1/2-\gamma \beta}$ for all $n\in\N$. This completes the proof of part (ii).

(iii): The assertion follows from (i)--(ii).

(iv): By the Marcinkiewicz--Zygmund strong law of large numbers, we have that $n^{r}(\widehat m_{u_n}-m)$ converges $\pr$-a.s.\ to zero for every $r<1/2$. So the assertion follows from part (iii).

(v): The classical Central Limit Theorem says that the law of $n^{1/2}(\widehat m_{u_n}-m)$ converges weakly to ${\cal N}_{0,\,s^2}$. So the assertion follows from Slutzky's lemma and part (iii).
{\hspace*{\fill}\proofendsign\par\bigskip}

%%%%%%%%%%%%%%%%%%%%%%%%%%%%%%%%%%%%%%%%%%%%%%%%%%%%%%%%%%%%%%%%
%%%%%%%%%%%%%%%%%%%%%%%%%%%%%%%%%%%%%%%%%%%%%%%%%%%%%%%%%%%%%%%%
%%%%%%%%%%%%%%%%%%%%%%%%%%%%%%%%%%%%%%%%%%%%%%%%%%%%%%%%%%%%%%%%

\subsection{Proof of Theorem \ref{main theorem - EPiE}}

(i): Analogously to (\ref{main theorem - proof - eq 10}), we obtain
\begin{eqnarray}\label{main theorem - proof - eq 20}
    {\cal R}_\rho({\cal N}_{n\widehat m_{u_n}(\omega),\,n\widehat s_{u_n}^2(\omega)})-{\cal R}_\rho(\widehat\mu_{u_n}^{\,*n}(\omega;\cdot))\,=\,\sqrt{n}\widehat s_{u_n}(\omega)\big({\cal R}_\rho({\cal N}_{0,1})-{\cal R}_\rho(\widehat{\mathfrak{m}}_n(\omega;\cdot)\big)
\end{eqnarray}
for all $\omega\in\Omega$, where $\widehat{\mathfrak{m}}_n(\omega;\cdot)$ denotes the law of the random variable $\widehat Z_n^\omega(\cdot):=(\widehat S_{n}^\omega(\cdot)-n\widehat m_{u_n}(\omega))/(\sqrt{n}\widehat s_{u_n}(\omega))$ for any random variable $\widehat S_n^\omega(\cdot)$ with distribution $\widehat\mu_{u_n}^{\,*n}(\omega;\cdot)$ and defined on some probability space $(\Omega^\omega,{\cal F}^\omega,\pr^\omega)$. For (\ref{main theorem - proof - eq 20}) notice that $\widehat\mu_{u_n}(\omega;\cdot)$ has mean $\widehat m_{u_n}(\omega)$ and standard deviation $\widehat s_{u_n}(\omega)$ for every fixed $\omega$.

First let $\lambda>3$. By the nonuniform Berry--Esséen inequality of Theorem \ref{Michels inequality}, we have
\begin{eqnarray}\label{main theorem - proof - eq 30}
    d_{\phi_\lambda}({\cal N}_{0,1},\widehat{\mathfrak{m}}_n(\omega;\cdot))
    & \le & C_\lambda\max\bigg\{\frac{\int \big|x-\int y\,\widehat\mu_{u_n}(\omega;dy)\big|^{3}\,\widehat\mu_{u_n}(\omega;dx)}{\big\{\int \big(x-\int y\,\widehat\mu_{u_n}(\omega;dy)\big)^{2}\,\widehat\mu_{u_n}(\omega;dx)\big\}^{3/2}};\nonumber\\
    & & \qquad\qquad\frac{\int \big|x-\int y\,\widehat\mu_{u_n}(\omega;dy)\big|^{\lambda}\,\widehat\mu_{u_n}(\omega;dx)}{\big\{\int \big(x-\int y\,\widehat\mu_{u_n}(\omega;dy)\big)^{2}\,\widehat\mu_{u_n}(\omega;dx)\big\}^{\lambda/2}}\bigg\}\,n^{-\gamma}
\end{eqnarray}
for all $n\in\N$, where $C_\lambda\in(0,\infty)$ is a universal constants depending only on $\lambda$ and being independent of $n$ and $\omega$. As a consequence of part (a) of Assumption \ref{main theorem} we have that $\int |x|^{\lambda}\,\widehat\mu_{u_n}(\omega;dx)=\frac{1}{u_n}\sum_{i=1}^{u_n}|Y_i|^\lambda$ converges to $\ex[|Y_1|^\lambda]$ for $\pr$-a.e.\ $\omega$. That is, %$\int \big|x-\int y\,\widehat\mu_{u_n}(\omega;dy)\big|^{\lambda}\,\widehat\mu_{u_n}(\omega;dx)$
the numerator of
\begin{equation}\label{main theorem - proof - eq 31}
    \frac{\int \big|x-\int y\,\widehat\mu_{u_n}(\omega;dy)\big|^{\lambda}\,\widehat\mu_{u_n}(\omega;dx)}{\big\{\int \big(x-\int y\,\widehat\mu_{u_n}(\omega;dy)\big)^{2}\,\widehat\mu_{u_n}(\omega;dx)\big\}^{\lambda/2}}
\end{equation}
is bounded above by an expression that converges to $2^{\lambda}\ex[|Y_1|^\lambda]$ for $\pr$-a.e.\ $\omega$. The denominator is nothing but $\widehat s_{u_n}(\omega)^{\lambda}$ and thus converges to $s^{\lambda}$ for $\pr$-a.e.\ $\omega$. That is, the expression in (\ref{main theorem - proof - eq 31}) converges to a positive constant for $\pr$-a.e.\ $\omega$. In the same way we obtain that
$$
    \frac{\int \big|x-\int y\,\widehat\mu_{u_n}(\omega;dy)\big|^{3}\,\widehat\mu_{u_n}(\omega;dx)}{\big\{\int \big(x-\int y\,\widehat\mu_{u_n}(\omega;dy)\big)^{2}\,\widehat\mu_{u_n}(\omega;dx)\big\}^{3/2}}
$$
converges to a positive constant for $\pr$-a.e.\ $\omega$. Together with (\ref{main theorem - proof - eq 20}), part (d) of Assumption \ref{basic assumption}, (\ref{main theorem - proof - eq 30}), and the $\pr$-a.s.\ convergence of $\widehat s_{u_n}$ to $s$, this implies
\begin{equation}\label{main theorem - proof - eq 32}
    n^{-1}\big({\cal R}_\rho({\cal N}_{n\widehat m_{u_n}(\omega),\,n\widehat s_{u_n}^2(\omega)})-{\cal R}_\rho(\widehat\mu_{u_n}^{\,*n}(\omega;\cdot)\big)
    %=\,{\cal O}_{\mbox{\rm \scriptsize{$\pr$-a.s.}}}(n^{-1/2-\gamma}).
    =\,{\cal O}(n^{-1/2-\gamma \beta})
\end{equation}
for $\pr$-a.e.\ $\omega$.

For $2<\lambda\le 3$ one can derive (\ref{main theorem - proof - eq 32}) in the same way, where on the right-hand side in (\ref{main theorem - proof - eq 30}) the expression  $\max\{\cdots\}$ has to be replaced by $\int
|x-\int y\,\widehat\mu_{u_n}(\omega;dy)\big|^{3}\widehat\mu_{u_n}(\omega;dx)/\{\int \big(x-\int y\,\widehat\mu_{u_n}(\omega;dy)\big)^{2}\widehat\mu_{u_n}(\omega;dx)\}^{3/2}$.  This completes the proof of part (i).

(ii): The assertion follows from (i)--(ii) of Theorem \ref{main theorem} and part (i) of Theorem \ref{main theorem - EPiE}.

(iii)-(iv): The assertions can be proven in the same way as the assertions (iv)--(v) of Theorem \ref{main theorem}; just replace part (iii) of Theorem \ref{main theorem} by part (ii) of Theorem \ref{main theorem - EPiE}.
{\hspace*{\fill}\proofendsign\par\bigskip}

%%%%%%%%%%%%%%%%%%%%%%%%%%%%%%%%%%%%%%%%%%%%%%%%%%%%%%%%%%%%%%%%
%%%%%%%%%%%%%%%%%%%%%%%%%%%%%%%%%%%%%%%%%%%%%%%%%%%%%%%%%%%%%%%%
%%%%%%%%%%%%%%%%%%%%%%%%%%%%%%%%%%%%%%%%%%%%%%%%%%%%%%%%%%%%%%%%

\subsection{Proof of Remark \ref{remark on measurability}}\label{proof of remark on measurability}

Let $\rho:L^p\rightarrow\R$ be a law-invariant coherent risk measure. First, Theorem 2.8 in \cite{Kraetschmeretal2014} ensures that the corresponding risk functional ${\cal R}_\rho:{\cal M}(L^p)\rightarrow\R$ is continuous for the $p$-weak topology ${\cal O}_{p\scriptsize{\mbox{-w}}}$. The latter is defined to the the coarsest topology on ${\cal M}(L^p)$ w.r.t.\ which each of the maps $\mu\mapsto\int f\,d\mu$, $f\in C_{\scriptsize{\mbox{b}}}^p$, is continuous, where $C_{\scriptsize{\mbox{b}}}^p$ is the set of all continuous functions $f:\R\rightarrow\R$ for which there exists a constant $C>0$ such that $|f(x)|\le C(1+|x|^p)$ for all $x\in\R$. According to Corollary A.45 in \cite{FoellmerSchied2011} the topological space $({\cal M}(L^p),{\cal O}_{p\scriptsize{\mbox{-w}}})$ is Polish. Second, the topology ${\cal O}_{p\scriptsize{\mbox{-w}}}$ is generated by the $L^p$-Wasserstein metric $d_{\scriptsize{\mbox{W}}_p}$ and the mapping ${\cal M}(L^p)\rightarrow{\cal M}(L^p)$, $\mu\mapsto\mu^{*n}$, is
$(d_{\scriptsize{\mbox{W}}_p},d_{\scriptsize{\mbox{W}}_p})$-continuous; see Lemma 8.6 in \cite{BickelFreedman1981}. Third, the mapping $\omega\mapsto\widehat\mu_{u_n}(\omega,\cdot)$ is $({\cal F},\sigma({\cal O}_{p\scriptsize{\mbox{-w}}}))$-measurable. Indeed, it is easily seen that the Borel $\sigma$-algebra $\sigma({\cal O}_{p\scriptsize{\mbox{-w}}})$ on ${\cal M}(L^p)$ is generated by the maps $\mu\mapsto\int fd\mu$, $f\in C_{\scriptsize{\mbox{b}}}^p$. So, for $({\cal F},\sigma({\cal O}_{p\scriptsize{\mbox{-w}}}))$-measurability of the mapping $\Omega\rightarrow{\cal M}(L^p)$, $\omega\mapsto\widehat\mu_{u_n}(\omega,\cdot)$, it suffices to show
\begin{equation}\label{proof of remark on measurability - eq}
    \Big(\int f(x)\,\widehat\mu_{u_n}(\cdot\,,dx)\Big)^{-1}(A)\in{\cal F}\qquad\mbox{ for all }A\in{\cal B}(\R)\mbox{ and }f\in C_{\scriptsize{\mbox{b}}}^p.
\end{equation}
Since $\widehat\mu_{u_n}(\omega,\cdot)$ is a probability kernel from $(\Omega,{\cal F})$ to $(\R,{\cal B}(\R))$, the mapping $\omega\mapsto\int f(x)\,\widehat\mu_{u_n}(\omega,dx)$ is $({\cal F},{\cal B}(\R))$-measurable for every $f\in C_{\rm b}^p$; see e.g.\ Lemma 1.41 in \cite{Kallenberg2002}. This gives (\ref{proof of remark on measurability - eq}). Altogether, we have shown that the mapping $\omega\mapsto{\cal R}_\rho(\widehat\mu_{u_n}^{\,*n}(\omega,\cdot))$ is $({\cal F},{\cal B}(\R))$-measurable.
{\hspace*{\fill}\proofendsign\par\bigskip}

%%%%%%%%%%%%%%%%%%%%%%%%%%%%%%%%%%%%%%%%%%%%%%%%%%%%%%%%%%%%%%%%
%%%%%%%%%%%%%%%%%%%%%%%%%%%%%%%%%%%%%%%%%%%%%%%%%%%%%%%%%%%%%%%%
%%%%%%%%%%%%%%%%%%%%%%%%%%%%%%%%%%%%%%%%%%%%%%%%%%%%%%%%%%%%%%%%
%%%%%%%%%%%%%%%%%%%%%%%%%%%%%%%%%%%%%%%%%%%%%%%%%%%%%%%%%%%%%%%%
%%%%%%%%%%%%%%%%%%%%%%%%%%%%%%%%%%%%%%%%%%%%%%%%%%%%%%%%%%%%%%%%

\appendix

\section{On the computation of $\widehat\mu_{u}^{\,*n}$ and ${\cal R}_\rho(\widehat\mu_{u}^{\,*n})$}\label{A remark on the recursive scheme}

In general the computation of the $n$-fold convolution $\widehat\mu_{u}^{\,*n}$ of $\widehat\mu_u$ is more or less impossible. However, in real applications the true $\mu$ has support in $h\N_0:=\{0,h,2h,\ldots\}$ for some fixed $h>0$, where $h$ represents the smallest monetary unit. We stress the fact that continuous distributions are in fact approximations for the {\em equidistant discrete} true single claim distribution, and not vice versa. So the empirical probability measure $\widehat\mu_u$ is concentrated on the equidistant grid $h\N_0$, too. In this case the estimated total claim distribution $\widehat{\mu}_{u}^{\,*n}$ can be computed with the help of the recursive scheme
\begin{eqnarray}
    \widehat\mu_{u}^{\,*n}[\{0\}] & = & \widehat\mu_{u}[\{0\}]^n \label{faltungsformel - eq 1} \\
    \widehat\mu_{u}^{\,*n}[\{jh\}] & = & \frac{1}{j\,\widehat\mu_{u}[\{0\}]}\sum_{\ell=1}^j((n+1)\ell-j)\,\widehat\mu_{u}[\{\ell h\}]\,\widehat\mu_{u}^{\,*n}[\{(j-\ell)h\}]\quad\mbox{ for }j\in\N,\qquad\label{faltungsformel - eq 2}
\end{eqnarray}
provided $\widehat\mu_{u}[\{0\}]>0$; see the discussion below. Note that $\widehat{\mu}_u$ as an empirical probability measure has bounded support. Therefore, the whole distribution $\widehat\mu_{u}^{\,*n}$ can be computed by the scheme (\ref{faltungsformel - eq 1})--(\ref{faltungsformel - eq 2}) in finitely many steps. In particular, the estimator
${\cal R}_\rho(\widehat\mu_{u}^{\,*n})$ can be computed in finitely many steps even for tail-dependent functionals ${\cal R}_\rho$ as, for instance, the one associated with the Average Value at Risk (introduced at the end of Section \ref{Sec DRM}).

To justify the scheme (\ref{faltungsformel - eq 1})--(\ref{faltungsformel - eq 2}) note that the empirical probability probability measure $\widehat\mu_u$ defined in (\ref{estimator for mu}) has the representation
$$
    \widehat\mu_u[\,\cdot\,]\,=\,\widehat p_u\,\widehat\nu_u[\,\cdot\,]+(1-\widehat p_u)\,\delta_0[\,\cdot\,],
$$
where $\widehat p_u:=\widehat\mu_u[(0,\infty)]$ is the mass of $\widehat\mu_u$ on $(0,\infty)$, and $\widehat\nu_u[\,\cdot\,]:=\widehat\mu_u[\,\cdot\,\cap (0,\infty)]/\widehat\mu_u[(0,\infty)]$ is the probability measure $\widehat\mu_u$ conditioned on $(0,\infty)$. It is easily seen that the $n$-fold convolution $\widehat\mu_u^{\,*n}$ coincides with the random convolution
$$
    \widehat\nu_u^{\,*\binv_{n,\widehat p_u}}[\,\cdot\,]\,:=\,\sum_{k=0}^n\widehat\nu_u^{\,*k}[\,\cdot\,]\,\binv_{n,\widehat p_u}[\{k\}]
$$
of $\widehat\nu_u$ w.r.t.\ the binomial distribution $\binv_{n,\widehat p_u}$ with parameters $n$ and $\widehat p_u$, i.e.
\begin{equation}\label{conv equals random conv}
    \widehat\mu_u^{\,*n}\,=\,\widehat\nu_u^{\,*\binv_{n,\widehat p_u}}.
\end{equation}
When $\widehat p_u<1$ and $\widehat\nu_u$ has support in $h\N:=\{h,2h,\ldots\}$ for some $h>0$, the random convolution $\widehat\nu_u^{\,*\binv_{n,\widehat p_u}}$ can be computed with the help of the Panjer recursion \cite{Panjer1981}:
\begin{eqnarray}
    \widehat{\nu}_{u}^{\,*\binv_{n,\widehat p_u}}[\{0\}] & = & \binv_{n,\widehat p_u}[\{0\}] \label{binomial faltungsformel - eq 1} \\
    \widehat{\nu}_{u}^{\,*\binv_{n,\widehat p_u}}[\{jh\}] & = & \frac{\widehat p_u/j}{1-\widehat p_u} \sum_{\ell=1}^j[(n+1)\ell-j]\,\widehat{\nu}_{u}[\{\ell h\}]\,\widehat{\nu}_{u}^{\,*\binv_{n,\widehat p_u}}[\{(j-\ell)h\}]\,\mbox{ for }j\in\N.\qquad\label{binomial faltungsformel - eq 2}
\end{eqnarray}
Since $1-\widehat p_u=\widehat\mu_u[\{0\}]$ and $\widehat p_u\widehat{\nu}_{u}[\{\ell h\}]=\widehat\mu_u[\{\ell h\}]$ for $\ell\in\N=\{1,2,\ldots\}$, the recursive scheme (\ref{faltungsformel - eq 1})--(\ref{faltungsformel - eq 2}) follows from (\ref{conv equals random conv})--(\ref{binomial faltungsformel - eq 2}).

%%%%%%%%%%%%%%%%%%%%%%%%%%%%%%%%%%%%%%%%%%%%%%%%%%%%%%%%%%%%%%%%
%%%%%%%%%%%%%%%%%%%%%%%%%%%%%%%%%%%%%%%%%%%%%%%%%%%%%%%%%%%%%%%%
%%%%%%%%%%%%%%%%%%%%%%%%%%%%%%%%%%%%%%%%%%%%%%%%%%%%%%%%%%%%%%%%
%%%%%%%%%%%%%%%%%%%%%%%%%%%%%%%%%%%%%%%%%%%%%%%%%%%%%%%%%%%%%%%%
%%%%%%%%%%%%%%%%%%%%%%%%%%%%%%%%%%%%%%%%%%%%%%%%%%%%%%%%%%%%%%%%
%%%%%%%%%%%%%%%%%%%%%%%%%%%%%%%%%%%%%%%%%%%%%%%%%%%%%%%%%%%%%%%%

%%%%%%%%%%%%%%%%%%%%%%%%%%%%%%%%%%%%%%%%%%%%%%%%%%%%%%%%%%%%%%%%
%%%%%%%%%%%%%%%%%%%%%%%%%%%%%%%%%%%%%%%%%%%%%%%%%%%%%%%%%%%%%%%%
%%%%%%%%%%%%%%%%%%%%%%%%%%%%%%%%%%%%%%%%%%%%%%%%%%%%%%%%%%%%%%%%

\end{document}